\documentclass[11pt]{amsart}
\usepackage{amsmath,amssymb,amsthm}
\usepackage[latin1]{inputenc}
\usepackage{version,tabularx,multicol}
\usepackage{graphicx,float,psfrag}
\usepackage{stmaryrd}

\newcommand{\reff}[1]{(\ref{#1})}

\theoremstyle{plain}
\newtheorem{theo}{Theorem}[section]
\newtheorem*{theo*}{Theorem}
\newtheorem{cor}[theo]{Corollary}
\newtheorem{prop}[theo]{Proposition}
\newtheorem{lem}[theo]{Lemma}

\theoremstyle{remark}
\newtheorem{rem}[theo]{Remark}

\newcommand{\caa}{{\mathcal A}}
\newcommand{\cb}{{\mathcal B}}
\newcommand{\cc}{{\mathcal C}}
\newcommand{\cd}{{\mathcal D}}

\newcommand{\ch}{{\mathcal H}}
\newcommand{\ci}{{\mathcal I}}

\newcommand{\cl}{{\mathcal L}}

\newcommand{\cm}{{\mathcal M}}

\newcommand{\ct}{{\mathcal T}}
\newcommand{\cu}{{\mathcal U}}

\newcommand{\E}{{\mathbb E}}

\newcommand{\N}{{\mathbb N}}
\renewcommand{\P}{{\mathbb P}}

\newcommand{\R}{{\mathbb R}}

\newcommand{\T}{{\mathbb T}}

\newcommand{\rA}{{\bar{\rm A}}}

\newcommand{\rp}{{\mathfrak{p}}}

\newcommand{\bt}{{\mathbf t}}
\newcommand{\tn}{{\rm T}_n}

\newcommand{\tnu}{{\rm T}_{n,1}}
\newcommand{\tnd}{{\rm T}_{n,2}}
\newcommand{\ttn}{\tilde {\rm T}_n}
\newcommand{\tnv}{{\rm T}_{n,v}}
\newcommand{\cnrk}{{\mathcal N}_{n,r,U_k}}
\newcommand{\cnrs}{{\mathcal N}_{n,r,s}}
\newcommand{\bs}{{\mathbf s}}
\newcommand{\bu}{{\mathbf u}}

\newcommand{\ind}{{\bf 1}}

\newcommand{\mrca}{\mathfrak{m}}

\newcommand{\Card}{{\rm Card}\;}

\newcommand{\norm}[1]{\mathop{\parallel\! #1 \! \parallel}\nolimits}
\newcommand{\val}[1]{\mathop{\left| #1 \right|}\nolimits}
\newcommand{\inv}[1]{\mathop{\frac{1}{ #1}}\nolimits}
\newcommand{\expp}[1]{\mathop {\mathrm{e}^{ #1}}}

\newcommand{\Var}{{\rm Var}\;}

\newcommand{\lb}{[\![} 
\newcommand{\rb}{]\!]}

\title{Cost functionals for large (uniform and simply generated) random trees}
\date{\today}

\author{Jean-François Delmas}
\address{
 Jean-Fran\c cois Delmas,
Université Paris-Est, Cermics (ENPC), F-77455 Marne-la-Vallée.}
\email{delmas@cermics.enpc.fr}

\author{Jean-Stéphane Dhersin}
\address{
Jean-Stéphane Dhersin,
Université Paris 13, Sorbonne Paris Cité, LAGA, CNRS (UMR 7539), 93430 Villetaneuse, France}
\email{dhersin@math.univ-paris13.fr}

\author{Marion Sciauveau}
\address{
Marion Sciauveau, 
Université Paris-Est, Cermics (ENPC), F-77455 Marne-la-Vallée.}
\email{marion.sciauveau@enpc.fr}

\begin{document}

\thanks{This work is partially supported by DIM RDMath IdF}



\begin{abstract} 
  Additive  tree  functionals  allow  to  represent  the  cost  of  many
  divide-and-conquer  algorithms. We  give an  invariance principle  for
  such tree  functionals for  the Catalan  model (random  tree uniformly
  distributed among the  full binary ordered trees with  given number of
  nodes) and for simply generated trees (including random tree uniformly
  distributed among  the ordered trees  with given number of  nodes). In
  the  Catalan model,  this relies  on the  natural embedding  of binary
  trees  into the  Brownian  excursion  and then  on  elementary $  L^2$
  computations. We recover results first  given by Fill and Kapur (2004)
  and then by  Fill and Janson (2009). In the  simply generated case, we
  use  convergence  of  conditioned Galton-Watson  towards  stable  Lévy
  trees, which provides less precise results but leads us to conjecture a
  different  phase transition  value  between  ``global'' and  ``local''
  regime. We also recover results first  given by Janson (2003 and 2016)
  in the quadratic case and give a generalization to the stable case.
\end{abstract}

\maketitle

\section{Introduction}
Trees have lots of applications in various fields such as computer science
for data structure or in  biology for genealogical or phylogenetic trees
of extant  species.  Related to  those applications, the study  of large
trees has  attracted some  attention. In this  paper, we  shall consider
asymptotics for additive functionals of large trees corresponding to the
Catalan model and some simply generated trees. 

\subsection{A finite measure indexed by a tree}
\label{sec:measure}
Let  $\T$  denote the  set  of  all  rooted  finite ordered  trees.   For
$\bt\in \T$, let $|\bt|$ be the the number of nodes of $\bt$; for a node
$v\in \bt$,  let $\bt_v$  denote the  sub-tree of  $\bt$ above  $v$ (see
\reff{eq:def-tv} in Section \ref{sec:def-tree} for a precise
definition).  We  consider  the following  unnormalized  non-negative  finite
measure $\caa_\bt$:
\begin{equation}
   \label{eq:def-A*}
\caa_{\bt} (f)=\sum_{v\in \bt}
|\bt_v|f\left(\frac{|\bt_v|}{|\bt|}\right),
\end{equation}
where $f$ is  a measurable real-valued function defined on  $ [0,1]$. We
are  interested in  the asymptotic  distribution of  $\caa_\bt(f)$ when
$\bt$ belongs to a certain class  of trees and $|\bt|$ goes to infinity.
We  shall consider  two classes  of trees:  the binary  trees (and  more
precisely the Catalan model) and some simply generated trees.

We  give some  examples related  to  the measure  $\caa_\bt$ which  are
commonly used  in the analysis  of trees.  In  what follows, for  a tree
$\bt\in \T$,  we denote by $\emptyset$  its root and by $d$  the usual graph
distance on $\bt$.  For $v,w \in \bt$, we say that $w$ is an ancestor of
$v$         and        write         $w\preccurlyeq        v$         if
$d(\emptyset,v)=d(\emptyset,w)+ d(w,v)$.  For $u,v\in \bt$, we denote by
$u\wedge v$, the most recent common ancestor of $u$ and $v$: $u\wedge v$
is  the  only  element  of   $\bt$  such  that:  $w\preccurlyeq  u$  and
$w\preccurlyeq v$ implies $w\preccurlyeq u \wedge v$.
\begin{itemize}
   \item       The       \textbf{total        path       length}    of $\bt$ is
     defined by  
     $P(\bt)=\sum_{w\in  \bt}   d(\emptyset,w)$. As
$d(\emptyset, w)=\sum_{v \in \bt}\ind_{\{v\preccurlyeq  w\}}-1$, we get:
$P(\bt)
=\sum_{v \in \bt}\sum_{w\in \bt} \ind_{\{v\preccurlyeq  w\}}-|\bt|
=\caa_\bt(1)- |\bt|$.

\item The  \textbf{shape functional} of  $\bt$ is defined by $\sum_{w\in
    \bt} \log(\bt_w)$. Notice that  $\sum_{w\in
    \bt} \log(\bt_w)= |\bt|^{-1} \caa_\bt (\log(x)/x) +
  |\bt|\log(|\bt|)$.  (The function $\log(x)/x$ will not be
  covered by the main results of this paper.)
   \item    The      \textbf{Wiener    index}    of    $\bt$     is    defined    by     $W(\bt)=\sum_{u,w\in  \bt} d(u,w)$.  Since 
\[
d(u,w)=\sum_{v\in \bt} (\ind_{\{v\preccurlyeq u\}} +
\ind_{\{v\preccurlyeq w\}} - 
      2\ind_{\{v \preccurlyeq u, \, v\preccurlyeq w\}}),
\]
we deduce that $W(\bt)=2 |\bt| \left( \caa_\bt(1)- \caa_\bt(x)
\right)$. 
\end{itemize}
In a nutshell,  for $\bt\in \T$,  we have:
\begin{equation}
   \label{eq:dist-A1}
\Big(P(\bt),\,  W(\bt)\Big)=\Big( \caa_\bt(1) - |\bt|, \, 2 |\bt| (
  \caa_\bt(1)- \caa_\bt(x) )\Big). 
\end{equation}

The measure $\caa_\bt$ is also related to other  additive functionals
in the particular case of binary  trees, see Section
\ref{sec:index}.

\subsection{Additive functionals and toll functions for binary trees}  
\label{sec:index}

Additive  functionals on  binary trees  allow to  represent the  cost of
algorithms  such  as   ``divide  and  conquer'',  see   Fill  and  Kapur
\cite{fk:ld}.  For $ \bt\in \T $ a  full binary tree, we shall denote by
$1$  (resp. $2$)  the left  (resp. right)  child of  the root.   Thus
$\bt_1$ (resp. $\bt_2$) will be the left (resp. right) sub-tree of the  root of $\bt$.   A functional $F$ on  binary trees is  called an
additive functional if it satisfies the following recurrence relation:
\begin{equation}
   \label{eq:sfe}
F(\bt)=F(\bt_1)+F(\bt_2)+b_{|\bt|},
\end{equation}
for all  trees $\bt$ such  that $|\bt|\geq   3$ and with
$F(\{\emptyset\})=b_1$. The 
given sequence $(b_n, {n\in \N^*})$ is called the toll function. Notice
that:
\begin{equation}
   \label{eq:F=+b}
F(\bt)=\sum_{v\in \bt} b_{|\bt_v|}.
\end{equation}
In  the particular  case where  the toll  function is  a power 
function, that is $b_n=n^\beta$ for $n\in \N^*$ and some $\beta>0$, we get
$F(\bt)=|\bt|^{-\beta+1}  \caa_\bt (x^{\beta-1})$.    In such  cases, the
asymptotic study of the measure $\caa_\bt$ will provide the asymptotic
of the additive functionals.\\

We  say that  $v\in \bt$  is  a leaf  if  $|\bt_v|=1$. We denote  by
$\cl(\bt)$     the    set     of    leaves     of    $\bt$     and, when
$|\bt|>1$, by
$\bt^*=\bt\setminus\cl(\bt)$  the tree  $\bt$  without  its leaves.   We
stress  that  the  additive  functional considered  in  \cite{fk:ld}  is
exactly
\begin{equation}
   \label{eq:tildeF}
\tilde F(\bt)=F(\bt^*)=\sum_{v\in \bt^*}
b_{|\bt_v^*|}. 
\end{equation}
However the  asymptotics will be  the same as the  one for $F$  when the
toll function is  a power function, see Remark \ref{rem:tildeF}.
We  complete the examples of the previous section for binary trees. 
\begin{itemize}
\item The  \textbf{Sackin index} (or  external path length) of  a tree
  $\bt$, used to study the balance of  the tree, is similar to the total
  path   length  of   $\bt$  when   one  considers   only  the   leaves:
  $S(\bt)=\sum_{w\in  \cl(\bt) }  d(\emptyset,w)$. Using  that for  a full
  binary  tree  we  have  $|\bt|=  2  |\cl(\bt)|  -1$,  we  deduce  that
  $2S(\bt)=\sum_{v\in \bt} |\bt_v|- 1=\caa_\bt(1) -1$.

\item The  \textbf{Colless index} of a  binary tree $\bt$ is  defined as
  $C(\bt)=\sum_{v\in  \bt^*}   |L_v -R_v|$, where $L_v=|\cl(\bt_{v1})|$
  (resp. $R_v=|\cl(\bt_{v2})|$) is the number of leaves of the left
  (resp. right) sub-tree above $v$.   Since 
  $\bt$      is      a      full      binary      tree,      we      get
  $2L_v-2R_v=|\bt_{v1}|-|\bt_{v2}|$  and
  $|\bt_{v1}|+|\bt_{v2}|=|\bt_v|-1$. 
  We obtain that $2C(\bt)=\sum_{v\in \bt} |\bt_v|- |\bt| - 2 \chi(\bt)$,
  with
\begin{equation}
   \label{eq:def-chi}
\chi(\bt)=\sum_{v\in  \bt^*}  \min(   |\bt_{v1}  |,|\bt_{v2}|).
 \end{equation}  
  That  is 
  $2C(\bt)=\caa_\bt(1)- |\bt| - 2 \chi(\bt)$.
\item The  \textbf{cophenetic index} of a  tree $\bt$ (which is  used in
  \cite{mrr}  to  study   the  balance  of  the  tree)   is  defined  by
  $\mathrm{Co}(\bt)=\sum_{u,w\in   \cl(\bt),\,  u\neq   w}  d(\emptyset,
  u\wedge                                                           w)$.
  Using   again   that  $\bt$   is   a   full   binary  tree,   we   get
  $4                     \mathrm{Co}(\bt)=
  4\sum_{v\in \bt}
  |\cl(\bt_v)|(|\cl(\bt_v)|-1)-4|\cl(\bt)|(|\cl(\bt)|-1)   =  \sum_{v\in
    \bt}       |\bt_v|^2        -       |\bt|^2        -       |\bt|+1$.
  That is $4 \mathrm{Co}(\bt)= |\bt| \caa_\bt (x) - |\bt|^2 - |\bt|+1$.
\end{itemize}
In a nutshell,  for $\bt\in \T$ full binary,  we have:
\begin{equation}
   \label{eq:SCCo=A}
\Big(2S(\bt), 2C(\bt), 4  \mathrm{Co}(\bt)\Big)= \Big(\caa_\bt(1) -1,\,
  \caa_\bt(1)- |\bt| - 2 \chi(\bt),\, |\bt| \caa_\bt (x) - |\bt|^2 - |\bt|+1  \Big).
\end{equation}

\subsection{Main results on the asymptotics of additive functionals in the Catalan model}
\label{sec:catalan}

We consider  the Catalan  model: let  $\tn$ be  a random  tree uniformly
distributed among the set of full binary ordered trees with $n$ internal
nodes (and thus $n+1$  leaves), which
has cardinal $C_n=(2n)!/[(n!^2) (n+1)]$. We have:
\[
\boxed{|\tn|=2n+1}.
\]
Recall that $\tn$  is a (full binary) Galton-Watson tree  (also known as
simply generated tree)  conditioned on having $n$ internal  nodes. It is
well    known,   see    Tak\`acs   \cite{t:thrrbt},    Aldous
\cite{a:crt2,a:III} and Janson
\cite{j:wisgrt},   that  $|\tn|^{-3/2}P(\tn)$   converges  in
distribution, as  $n$ goes  to infinity,  towards $2\int_0^1  B_s\, ds$,
where $B=(B_s,  s \in [0,1])$ is  the normalized positive Brownian excursion.
This result, see Corollary \ref{cor:GW}, can be seen as a consequence of
the convergence in  distribution of $\tn$ (in fact  the contour process)
properly  scaled  towards  the  Brownian continuum  tree  whose  contour
process  is  $B$,  see  \cite{a:crt2}  and  Duquesne
\cite{d:ltcpcgwt},      or       Duquesne      and       Le      Gall
\cite{dlg:rtlpsbp}   in   the  setting   of   Brownian
excursion.  For a combinatorial approach, which can be extended to other
families     of      trees,     see     also     Fill      and     Kapur
\cite{fk:rafust,fk:ttadrst}     or    Fill,
Flajolet and Kapur \cite{ffk:sahptr}.

In \cite{fk:ld}, the authors considered the toll functions $b_n=n^\beta$
with  $\beta>0$  and  they  proved  that with  a  suitable  scaling  the
corresponding                     additive                    functional
$F_\beta(\tn)=|\tn|^{-\beta+1}  \caa_{\tn}  (x^{\beta-1})$  converge  in
distribution to a  limit, say $Y_\beta$.  The  distribution of $Y_\beta$
is  characterized  by  its moments.   (In  \cite{f:dbsturpm,fk:ld},  the
authors  considered also  the  toll function  $b_n=\log(n)$.)  See  also
Janson and Chassaing \cite{jc:2004} for asymptotics of the Wiener index,
which  is a  consequence of  the  joint convergence  in distribution  of
$(\caa_{\tn}(1),  \caa_{\tn}(x))$  with  a suitable  scaling  and  Blum,
Fran\c{c}ois and Janson \cite{bfj:mv} for  the convergence of the Sackin
and Colless indexes.  In Theorem \ref{theo:principal} (take $\alpha=2$),
we   prove   that,   in   the  Catalan   model,   the   random   measure
$|\tn|^{-3/2}  \caa_{\tn}$  converges  weakly   a.s.,  as  $n$  goes  to
infinity,  to  a  random  measure   $2\Phi_B$,  built  on  the  Brownian
normalized excursion $B$, see \reff{eq:def-phi-h} with $h=B$.  Using the
notation $\tnv=(\tn)_v$  for $v\in \tn$,  this proves in  particular the
following a.s. convergence
\begin{equation}
   \label{eq:cv-intro}
|\tn|^{-3/2} \,\sum_{v\in \tn}
|\tnv|\, \, f\left(\frac{|\tnv |}{|\tn|}\right)
\,\xrightarrow[n\rightarrow+\infty]{\text{a.s.}}\,
  2\Phi_B(f),
\end{equation}
simultaneously for  all real-valued  continuous function $f$  defined on
$[0,1]$.   Notice that Theorem  \ref{theo:principal} is  more general  as the
convergences hold jointly for  all measurable real-valued functions $f$
defined  on  $[0,1]$  such  that   $f$  is  continuous  on  $(0,1]$  and
$\sup_{x\in (0,1]} x^a  |f(x)|$ is finite for some  $a<1/2$. Notice this
covers  the case  of toll  functions $b_n=n^\beta$  with $\beta>1/2$  in
\cite{fk:ld} which corresponds to the so called ``global''
regime.   The  limit $2  \Phi_B(x^{\beta-1})$  gives  a representation  of
$Y_\beta$  for $\beta>1/2$,  which, thanks  to Corollary  \ref{cor:cvZ},
corresponds when $\beta\geq  1$ to the one announced in  Fill and Janson
\cite{fj:plartslrvrt}, that is 
\[
\Phi_B(x^{\beta-1})=\inv{2} \beta (\beta-1)\int_{[0,1]^2} |t-s|^{\beta -2} \, m_B(s,t)
\,ds\,dt,
\]
where $m_B(s,t)=\inf_{u\in [s\wedge t, s\vee t]} B(u)$.
In the ``local'' regime, that  is
$\beta\in (0, 1/2]$, according to Corollary \ref{cor:cvZ}
and Lemma \ref{lem:subexcursion}, the convergence \reff{eq:cv-intro} is
not relevant as $\Phi_B(x^{\beta-1})=+\infty $ a.s.; see
\cite{fk:ld} for the relevant  normalization. 

The  proof  of  Theorem   \ref{theo:principal}  relies  on  the  natural
embedding of  $\tn$ into  the Brownian  excursion, see  \cite{a:III} and
Le Gall \cite{lg:93}, so
that   the   convergence  in   distribution   of   the  random   measure
$|\tn|^{-3/2}\caa_{\tn}$ or of the additive functionals $F_\beta$ (which
holds simultaneously for all $\beta>1/2$)  is then an a.s.  convergence.
We  also  give   the  fluctuations  for  this   a.s.   convergence,  see
Proposition   \ref{prop:fluctuation}.   In   Remark  \ref{rem:wsc},   we
provide, as  a direct  consequence of Theorem  \ref{theo:principal}, the
joint convergence of the total  length path, the Wiener, Sackin, Colless
and  cophenetic  indexes  defined   in  Sections  \ref{sec:measure}  and
\ref{sec:index}.

\begin{rem}
   \label{rem:rpm}
The method presented in this section based on the embedding of $\tn$ into
a  Brownian  excursion can not be extended directly to other models of
trees such as binary search trees, recursive trees or simply generated trees. 

Concerning  binary search  trees (or  random permutation  model or  Yule
trees),  see \cite{r:lds}  and \cite{r:q}
for the  convergence of the  external path length (which  corresponds in
our setting to the Sackin index), \cite{n:mono} for toll
function $b_n=n^\beta$,  \cite{n:wiener} for  the Wiener
index  (and  \cite{j:wisgrt}  for  simply  generated  trees),
\cite{bfj:mv}  (and  \cite{f:pc}
for   other  trees)   for   the  Sacking   and   Colless  indexes,   and
\cite{f:dbsturpm} for the shape function.

Concerning    recursive     trees,    see
\cite{m:ld,df:tpl} for the convergence
of the total path length 
and  \cite{n:wiener}  for  the  Wiener  index.   In  the
setting  of recursive  trees, then  \reff{eq:sfe} is  a stochastic  fixed
point  equation,   which  can   be  analyzed   using  the   approach  of
\cite{rr:cmra}.
\end{rem}

\begin{rem}
   \label{rem:fringe}
   One can replace the toll function  $b_{|\bt|} $ in \reff{eq:sfe} by a
   function of  the tree,  say $\textbf{b}(\bt)$.   For example,  if one
   consider $\textbf{b}(\bt)=\ind_{\{\bt=\bt_0\}}$, with $\bt_0$ a given
   tree, then the corresponding additive  functional gives the number of
   occurrence of the motif $\bt_0$.  The case of ``local'' toll function
   $\textbf{b}$ (with finite  support or fast decreasing  rate) has been
   considered  in  the  study  of fringe  trees,  see  \cite{a:afdgfrt},
   \cite{d:llsf,fgm:prbst}  for binary  search trees,  and \cite{j:ftgw}
   for simply generated  trees and \cite{hj:ft} for  binary search trees
   and recursive trees.

   See \cite{hn:pc}  for the  study of  the phase
   transition on asymptotics of additive functionals with toll functions
   $b_n=n^\beta$  on binary  search trees  between the  ``local'' regime
   (corresponding  to   $\beta\leq  1/2$)  and  the   ``global''  regime
   ($\beta>1/2$).  The same phase transition is observed for the Catalan
   model,  see  \cite{fk:ld}.  Our  main  result,  see
   Theorem  \ref{theo:principal}, concerns  specifically the  ``global''
   regime.
\end{rem}

\subsection{Main results on the asymptotics of additive functionals for simply generated trees}
\label{sec:simply}
We consider  a weight sequence $\rp=(\rp  (k), k\in \N)$ on  $\R_+$ with
generating  function $g_\rp$.   We assume  that $g_\rp$  has a  positive
radius of convergence, $g_\rp(0)=0$, $g_\rp \neq 0$ and  $\rp$ is generic, that is there
exists a positive  root to the equation  $g_\rp(q)=qg_\rp'(q)$.  A simply
generated  tree of  size $p\in  \N^*$ with  weight function  $\rp$ is  a
random tree $\tau^{(p)}$ such that the probability of $\tau^{(p)}$ to be
equal    to    $\bt$,    with     $|\bt|=p$,    is    proportional    to
$\prod_{v\in  \bt} \rp(k_v(\bt))$,  where  $k_v(\bt)$ is  the number  of
children   of   the  node   $v$   in   $\bt$.   According   to   Section
\ref{sec:simply-tree},  since  $g_\rp$  is   generic,  without  loss  of
generality  we  can   assume  that  $\rp$  is   a  critical  probability
($g_\rp(1)=g'_\rp(1)=1$),  so  that  $\tau^{(p)}$ is  distributed  as  a
Galton-Watson  (GW)  tree  $\tau$   with  offspring  distribution  $\rp$
conditioned to $|\tau|=p$. Global convergence of scaled GW trees
$\tau$ to Lévy trees has
been studied in  Le Gall and Le Jan \cite{lglj:bplp} and in 
\cite{dlg:rtlpsbp} using the convergence of contour process. 

Assume $\rp$ belongs  to the domain of attraction of  a symmetric stable
distribution of Laplace exponent $\psi(\lambda)= \kappa \lambda ^\gamma$
with  $\gamma\in  (1,2]$  and  $\kappa>0$.   Then,  the  convergence  of
$\tau^{(p)}$  properly  scaled  to   the  normalized  Lévy  trees  holds
according  to  \cite{d:ltcpcgwt}. This  result  is  recalled in  section
\ref{sec:td}. We  recall that the  normalized Lévy  tree is a  real tree
coded  by  the normalized  positive  excursion  of the  height  function
$H=(H(s), s \in [0,1])$.

 Under the hypothesis of Theorem \ref{theo:d},  there exists a
 sequence $(a_p, p\in \N^*)$ such that we  have the following
   convergence in distribution, see Corollary \ref{cor:cv-measure}:
\begin{equation}
   \label{eq:cv-tp}
\frac{a_p}{p^2} \,\sum_{v\in \tau^{(p)}}
|\tau ^ {(p)}_v|\, \, f\left(\frac{|\tau^{(p)}_v |}{p}\right)
\,\xrightarrow[p\rightarrow+\infty]{\text{(d)}}\,
  \Phi_H(f),
\end{equation}
simultaneously for  all real-valued  continuous function $f$  defined on
$[0,1]$.  The convergence \reff{eq:cv-tp} has to be understood along the
infinite  sub-sequence of  $p$  such that  $\P(|\tau|=p)>0$.  The  proof
relies  on  the  fact  that one  can  approximate  $\caa_\bt(x^k)$,  for
$k\in  \N^*$, by  an  elementary continuous  functional  of the  contour
process  of  $\bt$, see  Section  \ref{sec:approx}.   Then, we  use  the
convergence  of  the contour  process  of  $\tau^{(p)}$ to  the  contour
process  of  $H$ to  conclude.  We  also  provide  the first  moment  of
$\Phi_H(x^{\beta-1})$,  see  Lemma   \ref{lem:ZH}  and  conjecture  that
$\beta=1/\gamma$  corresponds  to  the   phase  transition  between  the
``global'' and ``local'' regime in this setting.

\begin{rem}
   \label{rem:sgt}
We make the following comments. 
\begin{itemize}
   \item Assume that $\rp$ has  finite variance, say $\sigma^2$. 
Then one can take
$a_p=\sqrt{p}$ and $H$ is equal to $(2/\sigma) B$
which corresponds to $\psi(\lambda)= \sigma^2\lambda ^2/2$. By scaling,
or using that the limit in Theorem \ref{theo:principal} does not depend
on $\alpha$, we deduce that $\Phi_{c B}=c \Phi_B$. We can then
rewrite \reff{eq:cv-tp} as:
\begin{equation}
   \label{eq:cv-tp-s2}
p ^{-3/2}  \,\sum_{v\in \tau^{(p)}}
|\tau ^ {(p)}_v|\, \, f\left(\frac{|\tau^{(p)}_v |}{p}\right)
\,\xrightarrow[p\rightarrow+\infty]{\text{(d)}}\,
  \frac{2}{\sigma} \Phi_B(f),
\end{equation}
where  the   convergence  holds   simultaneously  for   all  real-valued
continuous           function          $f$           defined          on
$[0,1]$     and      along     the     infinite      sub-sequence     of
$p$ such that $\P(|\tau|=p)>0$. 

\item If one consider the  binary offspring distribution $\rp$ such that
  $\rp(2)+\rp(0)=1$ (recall  that $1>\rp(0)>0$ by assumption),  one gets
  that $\tau ^{(2n+1)}$  is uniformly distributed among  the full binary
  trees  with   $n$  internal   nodes  (and   $n+1$  leaves),   that  is
  $\tau^{(2n+1)}$ is distributed as $\tn$, see the Catalan model studied
  in Section  \ref{sec:catalan}.  Take $\rp(0)=1/2$ to  get the critical
  case,  and notice  that $\sigma=1$  in \reff{eq:cv-tp-s2}.  The
  convergence  \reff{eq:cv-tp-s2},  with $p=2n+1$,  is then  a weaker  version of
  \reff{eq:cv-intro} (convergence in distribution instead of
  a.s. convergence, and continuous functions on $[0,1]$ instead of
  continuous functions on $(0,1]$ with possible blow up at $0+$).

\item   If   one   consider  the   (shifted)   geometric   distribution:
  $\rp(k)=q(1-q)^k$  for $k\in  \N$  with $q\in  (0,1)$,  one gets  that
  $\tau ^ {(p)}$ is uniformly distributed among the rooted ordered trees
  with $p$ nodes. Take $\rp(0)=1/2$ to get the critical case, and notice
  that $\sigma=2$ in \reff{eq:cv-tp-s2}.
\end{itemize}
\end{rem}

\subsection{Organization of the paper} 
Section \ref{sec:not} is  devoted to the definition of  the main objects
used  in  this  paper  (ordered  rooted  discrete  trees  using  Neveu's
formalism, real trees defined by a contour function, Brownian tree whose
contour function  is a Brownian  normalized excursion, the  embedding of
the discrete binary trees from the Catalan model into the Brownian tree,
and simply generated random trees). We present our main result about the
Catalan  model  in  Section \ref{sec:main-catalan}  on  the  convergence
\reff{eq:cv-intro},  see  Theorem   \ref{theo:principal}  and  Corollary
\ref{cor:cvZ}. (The  proofs are given in  Sections \ref{sec:premlim-lem}
and \ref{sec:proof-main}.)  The corresponding fluctuations are stated in
Proposition  \ref{prop:fluctuation}.  (The  proof  is  given in  Section
\ref{sec:proof-fluc}.)  Section \ref{sec:main-simply}  is devoted to the
main results concerning the convergence of $\caa_{\tau}$ when $\tau$ is
a  simply  generated  tree,  see  Corollaries  \ref{cor:cv-measure}  and
\ref{cor:GW}.     (Their     proofs    are    provided     in    Section
\ref{sec:proof-cv-measure}.) Some technical  results are gathered  in Section
\ref{sec:appendix}.

\section{Notations and a preliminary result}
\label{sec:not}
Let $I$ be an interval of $\R$ with positive Lebesgue measure. We denote
by $\cb(I)$ the set of  real-valued measurable functions defined on $I$.
We denote by $\cc(I)$ (resp.   $\cc_+(I)$) the set of real-valued (resp.
non-negative) continuous functions defined on  $I$. For $f\in \cb(I)$ we
denote by $\norm{f}_\infty $ the supremum norm and by 
 $\norm{f}_{\text{esssup}} $
the essential supremum of $|f|$ over
$I$. The two supremums coincide when $f$ is continuous. 

\subsection{Ordered rooted discrete trees} 
\label{sec:def-tree}

We recall  Neveu's formalism \cite{n:apghw} for  ordered rooted discrete
trees,    which    we    shall    simply    call    trees.     We    set
$\cu=\bigcup  _{n\ge  0}{(\N^*)^n}$  the  set  of  finite  sequences  of
positive  integers with  the  convention $(\N^*)^0=\{\emptyset\}$.   For
$n\geq 0$ and  $u\in (\N^*)^n\subset \cu$, we set $|u|=n$  the length of
$u$.  Let $u,v\in \cu$.  We denote
by $uv$ the concatenation of the two sequences, with the convention that
$uv=u$ if $v=\emptyset$ and $uv=v$ if $u=\emptyset$.  We say that $v$ is
an ancestor  of $u$ (in  a large sense)  and write $v\preccurlyeq  u$ if
there  exists $w\in  \cu$ such  that $u=vw$.   If $v\preccurlyeq  u$ and
$v\neq u$, then we shall write $v\prec  u$.  The set of ancestors of $u$
is  the set  $ \rA_u=\{v\in  \cu; v\preccurlyeq  u\}$.  The  most recent
common ancestor of a subset $\bs$ of $ \cu$, denoted by $\mrca(\bs)$, is
the  unique  element $v$  of  $\bigcap_{u\in  \bs} \rA_u$  with  maximal
length.  We consider the lexicographic  order on $\cu$: for $u,v\in\cu$,
we  set $v<u$  either if  $v\prec u$  or if  $v=wjv'$ and  $u=wiu'$ with
$w=\mrca(\{v,u\})$, $u,u'\in \cu$ and $j<i$ for some $i,j\in\N^*$.

\bigskip

A tree $\bt$ is a subset of $\cu$ that satisfies:
\begin{itemize}
\item $\emptyset\in\bt$,
\item If  $u\in\bt$, then $\rA_u\subset \bt$. 
\item For every $u\in \bt$, there exists 
  $k_u(\bt)\in \N$ such that, for every  $i\in \N^*$,  $ui\in \bt$ if
  and only if  $1\leq i\leq k_u(\bt)$. 
\end{itemize}

Let  $u\in  \bt$.  The  integer  $k_u(\bt)$  represents  the  number  of
offsprings of the node $u$.  The node $u$ is called a  leaf
(resp. internal  node) if $k_u(\bt)=0$  (resp. $k_u(\bt)>0$). 
The  node $\emptyset$  is called  the root  of $\bt$.   We define  the
sub-tree $\bt_u\in \T$ of $\bt$ ``above'' $u$ as:
\begin{equation}
   \label{eq:def-tv}
\bt_u=\{v\in\cu,\ uv\in\bt\}.
\end{equation}
We denote by $|\bt|=\Card(\bt)$ the number  of nodes of $\bt$ and we say
that $\bt$ is finite if $|\bt|<+\infty  $.  Let $d_\bt$ denote the usual
graph    distance     on    $\bt$.      In    particular,     we    have
$d_\bt(\emptyset,u)=|u|$ for $u\in  \bt$. When the context  is clear, we
shall write $d$ for $d_\bt$.

We   denote   by   $\T   $   the    set   of   finite   trees   and   by
$\T^{(p)}=\{\bt\in \T,  \, | \bt|=p\}$ the  set of trees with  $p$ nodes,
for $p\in \N^*$.  Let us recall that, for a tree $\bt\in \T$, we have
\begin{equation}\label{eq:sum_k}
\sum_{u\in\bt}k_u(\bt)=|\bt|-1.
\end{equation}

\subsection{Real trees}
\label{sec:real-t}
We recall the definition of a real tree, see \cite{e:prt}.
A real tree is a metric space $( \ct,d)$ which satisfies the
following two properties for every $x,y\in \ct$:
 \begin{enumerate}
  \item[(i)] There exists a unique isometric map $f_{x,y}$ from $[0,d(x,y)]$ into $ \ct$ such that $f_{x,y}(0)=x$ and $f_{x,y}(d(x,y))=y$.
  \item[(ii)] If $\phi$ is a continuous injective map from $[0,1]$ into
    $ \ct$ such that $\phi(0)=x$ and $\phi(1)=y$, then we have
    $\phi([0,1])=f_{x,y}([0,d(x,y)])$. 
 \end{enumerate}
Equivalently,  a metric space  $( \ct,d)$ is a
real tree if and only if $ \ct$ is connected and $d$ satisfies the four point condition: 
\[
d(s,t)+d(x,y)\leq  \max(d(s,x)+d(t,y), d(s,y)+d(t,x)) \quad\textrm{for
  all} \quad s,t,x,y\in
     \ct.
\]

A rooted real tree is a real tree $( \ct,d)$ with a distinguished element
$\emptyset$ called the root.
 For  $x,y\in\ct$, we denote by
$\lb x,y\rb$ the range of the map $f_{x,y}$ described above. Let
$x,y\in\ct$. We denote by $x\wedge y$ their most recent common ancestor
which is the only $z\in \ct$ such that $\lb \emptyset, z \rb=\lb
\emptyset, x \rb \bigcap \lb \emptyset, y \rb$. 
The out-degree $d_x(\ct)$ of $x$ is the
 number of connected components of $\ct\backslash\{x\}$ which do not
 contain the root. We say $x$ is a
 leaf (resp. branching point) if  $d_x(\ct)=0$ (resp. $d_x(\ct)\geq
 2$). We say $\ct$ is binary if $d_x(\ct)\in \{0, 1, 2\}$ for all $x\in
 \ct$.\\

For  $h\in \cc_+([0,1])$, we
define its minimum over the interval with bounds $s,t\in [0,1]$:
\begin{equation}
   \label{eq:def-m}
m_h(s,t)=\inf_{u\in [s\wedge t, s\vee t]} h(u).
 \end{equation} 
 We shall  also use   the length  of the excursion  of $h$
 above level $r$ straddling $s$ defined by:
\begin{equation}
   \label{eq:def-srs}
\sigma_{r,s}(h)=\int_0^1 dt\,  \ind_{\{\min_h(s,t)\geq  r\}}.
\end{equation}
For $\beta> 0$, we set:
\begin{equation}
   \label{eq:zb}
Z_\beta^h=\int_{0}^{1}ds\int_{0}^{h(s)}dr\;\sigma_{r,s}(h)^{\beta-1}.
\end{equation}

Let  $h\in \cc_+([0,1])$ be  such
that    $m_h(0,1)=0$.    For    every    $x,y\in    [0,1]$,   we    set
$d_{h}(x,y)=h(x)+h(y)-2m_{h}(x,y)$.  It  is easy to check  that $d_h$ is
symmetric  and   satisfies  the   triangle  inequality.    The  relation
$\sim_{h}$           defined          on           $[0,1]^2$          by
$ x\sim_{h}  y \Leftrightarrow  d_h(x,y)=0$ is an  equivalence relation.
Let $ \ct_{h}=[0,1]/  \sim_ h$ be the corresponding  quotient space. The
function $d_h$ on  $[0,1]^2$ induces a function on $  \ct_h^2$, which we
still denoted by $d_h$, and which is  a distance on $ \ct_h$.  It is not
difficult  to check  that  $(  \ct_{h}, d_h)$  is  then  a compact  real
tree. We denote by ${\bf{p}}_h$ the canonical projection from $[0,1]$ into
$ \ct_h$.  Thus,  the metric space $( \ct_{h},d_{h})$   can  be  viewed   as  a   rooted  real  tree   by  setting
$\emptyset={\bf p}_h(0)$. The image of the Lebesgue measure on $[0,1]$ by
${\bf{p}}_h$ is a measure $\mu_h$ on $\ct_h$. 

\subsection{The Brownian continuum random tree $\ct$}
\label{sec:Te}

Let  $B=(B_t,  {0\le  t\le1})  $ be  a positive normalized  Brownian  excursion.
Informally, $B$ is just a linear standard Brownian path started from the
origin and conditioned  to stay positive on $(0,1)$ and  to come back to
$0$ at time  $1$.  For $\alpha>0$, let $e=\sqrt{2/\alpha} \,  B$ and let
$\ct_e$  denote the  associated  real tree  called  Brownian continuum
random tree.   (We recall the associated 
branching mechanism is $\psi(\lambda)=\alpha \lambda^2$.) The
continuum  random  tree  introduced   in  \cite{a:crt1}  corresponds  to
$\alpha=1/2$ and the Brownian tree associated to the normalized Brownian
excursion  corresponds  to  $\alpha=2$.   We shall  keep  the  parameter
$\alpha$ so that the two previous cases are easy to read on the results.
See  \cite{lg:rta}  for  properties  of the  Brownian  continuum  random
tree.  In particular  $\mu_e(dx)$-a.s.  $x$  is a  leaf  and a.s. $\ct_e$  is
binary.

We shall forget  to stress the dependence in $e$  in the notations, when
there is no ambiguity, so that for example we simply write $\ct$, $\mu$,
$\sigma_{r,s}$   and  $Z_\beta$   for  respectively   $\ct_e$,  $\mu_e$,
$\sigma_{r,s}(e)$ which is defined  in \reff{eq:def-srs} and $Z_\beta^e$
which is  defined in \reff{eq:zb}.  For  $r\geq 0$ and $s\in  [0,1]$, we
also have:
\[
\sigma_{r,s}= \mu(x\in \ct, \, d(\emptyset, x\wedge
{\bf{p}}(s) )\geq r)) ,
\]
which is the mass of the sub-tree of $\ct$ containing
${\bf{p}}(s)$ and at distance $r$ from
the root. 

The next result is a consequence of 
Lemma \ref{lem:ZH} in Section \ref{sec:main-simply} (with $H=e$, $\gamma=2$ and
$\kappa=\alpha$). 
\begin{lem}
\label{lem:subexcursion}
We have that a.s.  for all $ 1/2\geq \beta>0$,   $Z_\beta=+\infty
$. We have that a.s.  for all $\beta>1/2$,  $Z_\beta$ is finite and 
\begin{equation}
   \label{eq:EZ}
\E\left[Z_\beta\right]=\inv{2\sqrt{\alpha}}\, \frac 
{\Gamma\left(\beta-\inv{2}\right)}{\Gamma(\beta)}\cdot
\end{equation}
We also have the representation formulas $Z_1= \int_{0}^{1}e(s)\, ds$ and for $\beta>1$:
\begin{equation}
   \label{eq:Z=bis}
Z_\beta= \inv{2} \beta (\beta-1)\int_{[0,1]^2} |t-s|^{\beta -2} \, m(s,t)
\,ds\,dt.
\end{equation}
\end{lem}
All the moments of $Z_\beta$, for $\beta>1/2$, are given in \cite{fk:ld}, thanks
to  the   identification  provided  by  Corollary   \ref{cor:cvZ}.   The
representation formula for $Z_\beta$ is  motivated by the formulation of
our    Corollary    \ref{cor:cvZ}     given    in    \cite{fk:ld}    and
\cite{fj:plartslrvrt}.

\subsection{The discrete binary tree from the Brownian tree}
\label{sec:T[n]}

A marked tree $\tilde \bt=(\bt, (h_v,  v\in \bt))$ is a tree $\bt\in \T$
with a  label on  each node. The  label $h_v\in (0,  +\infty )$  will be
interpreted as the  length of the branch from below  $v$. (Notice, there
is a branch  below the root.) We define the  concatenation of two marked
trees $\tilde  \bt ^{(i)}=(\bt^{(i)}, (h_v^{(i)}, v\in  \bt^{(i)}))$ with
$i\in            \{1,2\}$             and            $r>0$            as
  $\tilde   \bt=[\tilde   \bt^{(1)},    \tilde   \bt^{(2)};   r]$   with
  $\bt=\{\emptyset\} \bigcup _{i=1}^2  \{v=iu, \, u\in \bt^{(i)}\}$ and
for $v\in \bt$, we have $h_v=r$ if $v=\emptyset$ and $h_v=h_{u}^{(i)}$ if $v=iu$
with $u\in \bt^{(i)}$ and $i\in \{1, 2\}$.\\

Let $g\in  \cc_+([0,1])$ be such that  $\ct_g$ is binary. Let  $n\in \N$
and       $0<t_1<        \cdots       <t_{n+1}<1$        such       that
$ ({\bf  p}_g(t_k), 1\leq  k\leq n+1)$ are  $n+1$ distinct  leaves.  Set
$G_n=(g;     t_1,      \ldots,     t_{n+1})$.      We      denote     by
$ \ct_g(G_n)=\bigcup _{k=1}^{n+1}  \lb \emptyset,{\bf{p}}_g(t_k)\rb$ the
random     real     tree     spanned     by     the     $n+1$     leaves
${\bf{p}}_g(t_1),\dots,{\bf{p}}_g(t_{n+1})$  with root  $\emptyset$.  We
define      recursively       the      associated       marked      tree
$\tilde  \bt(G_n)=(\bt(G_n),  (h_{n,v}(G_n),   v\in  \bt(G_n)))$,  where
intuitively $\bt(G_n)$  is similar to  $\ct_g(G_n)$ but with  the branch
lengths equal to  1 and no branch below the  root, and $h_{n,v}(G_n)$ is
the length of the branch in $\ct_g(G_n)$ below the node corresponding to
$v\in    \bt(G_n)$.     More    precisely,    for    $n=0$,    we    set
$\bt(G_0)=\{\emptyset\}$    and   $h_{0,\emptyset}(G_0)=g(t_1)$.     Let
$n\geq       1$.       Since       $\ct_g$      is       binary      and
$( {\bf  p}_g(t_k), 1\leq k\leq  n+1)$ are $n+1$ distinct  leaves, there
exists    a    unique   $s\in    (t_1,    t_{n+1})$    and   a    unique
$\ell\in  \{1,  \ldots,  n\}$  such that  $g(s)=m_g(t_1,  t_{n+1})$  and
$t_\ell<s<t_{\ell+1}$.                      We                    define
$g_1(t)=(g(t)-           g(s))\ind_{[t_1,          s]}(t)$           and
$g_2(t)=(g(t)- g(s))\ind_{[s,t_{n+1}]}(t)$.  Notice  that $\ct_{g_i}$ is
binary   and   $   ({\bf    p}_g(t_k),   1\leq   k\leq   \ell)$   (resp.
$   ({\bf  p}_g(t_k),   \ell+1\leq  k\leq   n+1)$)  are   $\ell$  (resp.
$n-\ell+1 $)  distinct leaves  of $\ct_{g_1}$ (resp.   $\ct_{g_2}$). Set
$G_{\ell-1}'=(g_1;        t_1,        \ldots,        t_\ell)$        and
$G_{n-\ell}''=(g_2;   t_{\ell+1},    \ldots,   t_{n+1})$    and   define
$\tilde          \bt(G_n)$         as          the         concatenation
$[\bt(G'_{\ell-1}), \bt (G''_{n-\ell});
g(s)]$. 
\\

Let $e$ be the Brownian  excursion defined in Section \ref{sec:Te}.  Let
$(U_n,  n\in\N^{*})$  be  a  sequence of  independent  random  variables
uniform    on   $[0,1]$,    independent   of    $e$.    In    particular
$({\bf  p}(U_n), n\in  \N^*)$ are  a.s. distinct  leaves of  $\ct$.  Let
$(U_{1,n},  \ldots, U_{n+1,n})$  be  the a.s.  increasing reordering  of
$(U_1,           \ldots,           U_{n+1})$           and           set
$G_n=(e; (U_{1,n},  \ldots, U_{n+1,n}))$. We  write $\ct_{[n]}=\ct(G_n)$
the    random    real    tree    spanned    by    the    $n+1$    leaves
${\bf{p}}(U_1),\dots,{\bf{p}}(U_{n+1})$     and     the     root     and
$\ttn=(\tn; (h_{n,v}, v\in \tn))=\tilde \bt (G_n)$ the associated marked
tree. For  $1\leq k\leq n+1$,  we denote by  $u(U_k)$ the leaf  in $\tn$
corresponding  to the  leaf ${\bf  p}(U_k)$ in  $\ct_{[n]}$. See  Figure
\reff{fig:treeT0} for an example with $n=4$. It is well known that $\tn$
is  uniform  among the  discrete  full  binary  ordered trees  with  $n$
internal
nodes. \\

\begin{figure}[ht]
\begin{center}
\includegraphics[height=6cm]{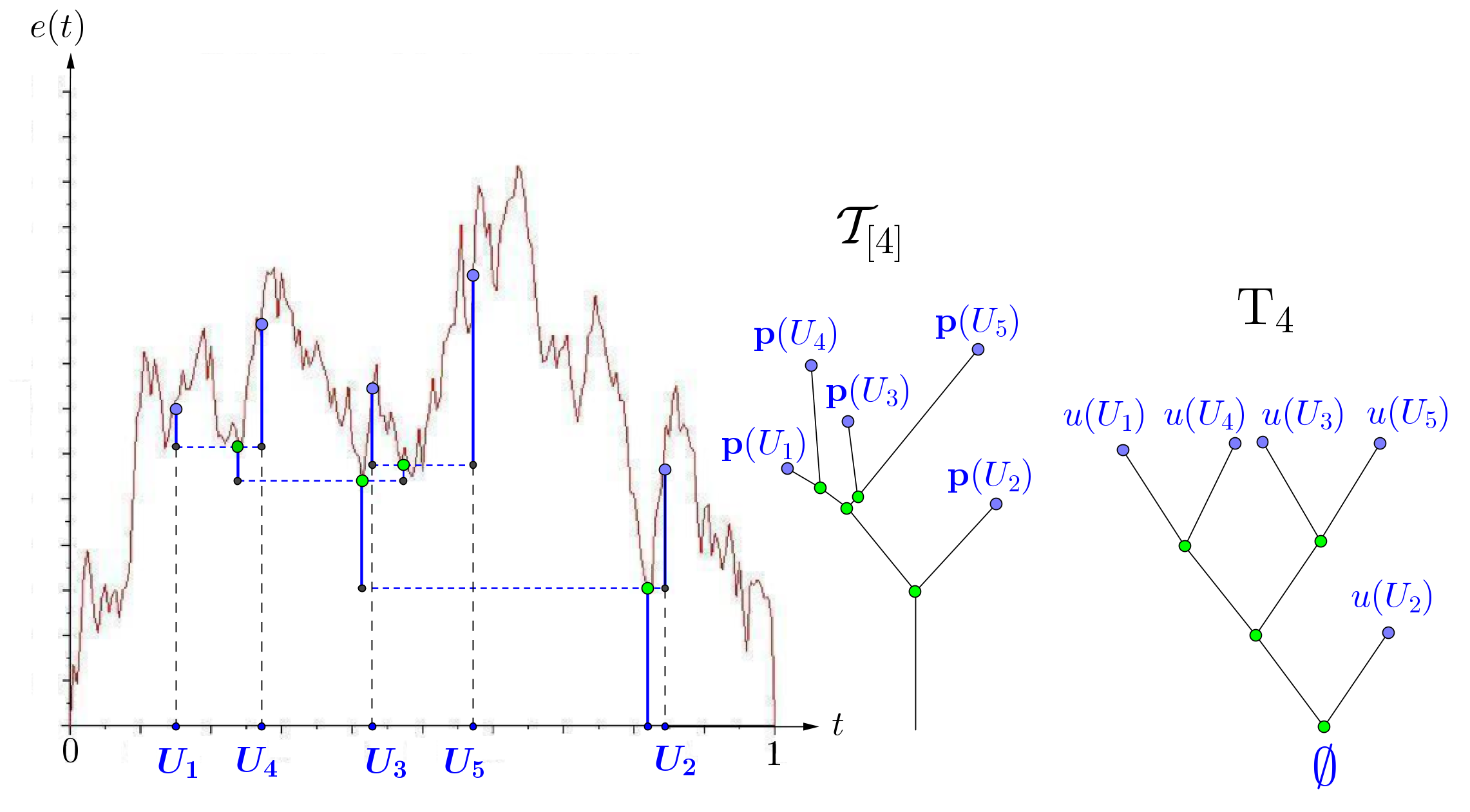}
\end{center}
\caption{The  Brownian excursion,  $ \ct_{[n]}$  and 
  $\tn$ (for  
  $n=4$).}
\label{fig:treeT0}
\end{figure}

\subsection{Simply generated random tree}
\label{sec:simply-tree}
We consider a weight sequence $\rp=(\rp (k), k\in \N)$ of
non-negative real numbers such
that $\sum_{k\in \N} \rp(k)>\rp(1)+\rp(0)$ and $\rp(0)>0$.
 For $\bt\in \T$, we define its
weight as:
\[
w(\bt)=\prod_{v\in \bt} \rp(k_v(\bt)).
\]
We set $w(\T^{(p)})=\sum_{\bt \in \T^{(p)}} w(\bt)$. For $p\in \N^*$
such that $w(\T^{(p)})>0$, a simply generated tree taking
values in $\T^{(p)}$ with weight sequence $\rp$ is a $\T^{(p)}$-random variable
$\tau^{(p)}$ whose distribution is characterized by, for all $\bt\in \T^{(p)}$:
\[
\P(\tau^{(p)}=\bt)=\frac{w(\bt)}{w(\T^{(p)})}\cdot
\]

Let    $g_\rp$     be    the     generating    function     of    $\rp$:
$g_\rp(\theta)=\sum_{k\in \N} \theta^k \rp(k)$  for $\theta>0$. From now
on,  we assume  there  exists $\theta>0$  such  that $g_\rp(\theta)$  is
finite.  For  $q>0$ such  that $g_\rp(q)<+\infty $,  let $\rp_q$  be the
probability       distribution       with      generating       function
$\theta\mapsto g_\rp(q\theta)/g_\rp(q)$.   According to \cite{k:gwpctp}
see also \cite{ad:llcGWcc}, 
the  distribution of  the  GW tree  $\tau$  with offspring  distribution
$\rp_q$  conditioned   on  $   \{|\tau|=p\}$  is  the   distribution  of
$\tau^{(p)}$ and thus does not depend on $q$.  It is easy to check there
exists  at  most   one  positive  root,  say  $q_\rp$,   of  the  equation
$g_\rp(q)=qg_\rp'(q)$.   We say  that  $\rp$ is  generic  (for the  total
progeny)  if  such  root  $q_\rp$ exists  and  non-generic  otherwise.   In
particular,  all   weight sequences such that 
there exists $q>0$ with
$g_\rp(q)$ finite and  $g_\rp(q)<qg_\rp'(q)$ (that is $\rp_q$ is a
super-critical offspring distribution), are generic. 

From now on, we shall assume that $\rp$ is generic. Without  loss of
generality, by replacing $\rp$ by the probability distribution with
generating function $\theta\mapsto g_\rp(q_\rp \theta)/g_\rp(q_\rp)$, 
we will assume that $\rp$ is a critical probability distribution, that
is:
\[
\sum_{k\in     \N}     \rp(k)=\sum_{k\in     \N}     k\rp(k)=1.
\]
We recall that $\tau^{(p)}$ is distributed  as a critical GW tree $\tau$
with offspring distribution $\rp$  conditioned on $\{|\tau|=p\}$, as for
all finite tree $\bt$, $\P(\tau=\bt)=w(\bt)$.

Local limits for  critical GW trees conditioned on having  a large total
progeny go back to \cite{k:gwpctp}  for the generic case (infinite spine
case) and \cite{j:sgtcgwt} for the non-generic case (condensation case),
see also  \cite{ad:llcGWcc,ad:llcGWisc} and  reference therein  for more
general  conditionings. Scaling  limits  or global  limits  for GW  tree
conditioned  on  having  a  large  total progeny  have  been  studied  in
\cite{dlg:rtlpsbp} for forests  (that is collection of GW  trees) and in
\cite{d:ltcpcgwt,k:spthd}  for  critical  GW   tree  in  the  domain  of
attraction  of  Lévy trees,  see  also  \cite{k:ipgw} for  more  general
conditioning of GW trees and \cite{k:ltnggw} for non-generic cases.

\section{Main results}

For $\bt\in \T$, we define the unnormalized measure $\caa_\bt$ on
$[0,1]$ by: 
\[
\caa_{\bt} (f)=\sum_{v\in \bt}
|\bt_v|f\left(\frac{|\bt_v|}{|\bt|}\right), \quad f\in \cc([0,1]).
\]
For  $h\in \cc_+([0,1 ])$,  we also consider 
the random measure $\Phi_h$  on $[0,1]$ defined by:
\begin{equation}
   \label{eq:def-phi-h}
\Phi_h(f)=  \int_{0}^{1}ds\int_{0}^{h(s)}dr\;
f(\sigma_{r,s}(h)), \quad f\in \cb([0,1]).
\end{equation}
We endow the space of non-negative finite measures on $[0,1]$ with the topology of
the weak converge. 

\subsection{Catalan model}
\label{sec:main-catalan}
Let $\alpha>0$ and recall $e=\sqrt{2/\alpha}\, B$, where $B=(B_t, t\in
[0,1])$  denotes the normalized Brownian excursion. 
We also recall that the discrete binary tree $\tn$, defined in Section
\ref{sec:T[n]} from the Brownian tree $\ct_e$, is uniformly distributed
among the full ordered rooted binary trees with $n$ internal nodes. In
particular, we have $|\tn|=2n+1$. 
For $n\in \N^*$, we define the weighted random measure $A_n$ on
$[0,1]$ defined by $A_n=|\tn|^{-3/2} \, \caa_{\tn}$, that is
for $f\in  \cb([0,1])$:
\begin{equation}
   \label{eq:def-An}
A_n(f)=|\tn|^{-3/2} \,\sum_{v\in \tn}
|\tnv|\, \, f\left(\frac{|\tnv |}{|\tn|}\right),
\end{equation}
where $\tnv=(\tn)_v$ is the sub-tree of $\tn$ ``above'' $v$. Notice that $A_n(\{0\})=0$.
The next result is proved in
Section \ref{sec:proof-main}.

\begin{theo}
   \label{theo:principal}
   We have  that a.s. for  all $f\in \cb([0,1])$, continuous  on $(0,1]$
   and  such that  $\lim_{x\rightarrow  0+}  \, x^{a}f(x)  =  0$ for  some
   $a\in [0,1/2)$:
\[
A_n(f)
\,\xrightarrow[n\rightarrow+\infty]{}\,
\sqrt{2\alpha}\,  \Phi_e(f).
\]
\end{theo}
We deduce from this Theorem that $(A_n, n\in \N^*)$ converges a.s. for
the weak topology towards $\sqrt{2\alpha}\, \Phi_e$. 

  By  convention, for $a\in \R$, we  denote  the  function
$x\mapsto x^a\ind_{(0,1]}(x)$ defined on $[0,1]$   by $x^a$. 
According to Lemma \ref{lem:subexcursion}, the random variable
$Z_\beta=\Phi_e(x^{\beta-1})$, see definition \reff{eq:zb}, is a.s.  finite (resp.  infinite) if $\beta>1/2$
(resp.  $0<\beta\leq  1/2$). We deduce  the following convergence from
Theorem \ref{theo:principal}. 
\begin{cor}
\label{cor:cvZ}
We have that a.s.  for all $\beta> 0$,
 \[
\lim_{n\rightarrow+\infty } |\tn|^{-(\beta+\frac{1}{2})}\sum_{v\in \tn}|\tnv|^{\beta}= \sqrt{2 \alpha} \, Z_\beta. 
\]
\end{cor}

\begin{proof}
Notice that 
$|\tn|^{-(\beta+\frac{1}{2})}\sum_{v\in
  \tn}|\tnv|^{\beta}=A_n(x^{\beta-1})$. 
   For $\beta>1/2$, the Corollary is then a direct consequence of
  Theorem \ref{theo:principal} with $f=x^{\beta-1}$.
We  now consider   the  case   $1/2\geq
  \beta>0$.  Let $c>0$. Using  Theorem \ref{theo:principal}, we have that
  a.s.:
\[
\liminf_{n\rightarrow+\infty }  A_n( x^{\beta-1}) 
\geq 
\lim_{n\rightarrow+\infty }  A_n(c\wedge x^{\beta-1}) 
= \sqrt{2 \alpha}\, \Phi_e(c\wedge x^{\beta-1}) .
\]
Letting   $c$   goes   to   infinity,    and   using   that,    by   Lemma
\ref{lem:subexcursion}, $\Phi_e(x^{\beta-1}) =Z_\beta=+\infty  $ a.s., we
get                               that                              a.s.
$\liminf_{n\rightarrow+\infty } A_n( x^{\beta-1}) \geq \sqrt{2 \alpha}\,
Z_\beta=+\infty $.  Then use a monotonicity argument in $\beta$ to deduce
the results holds a.s. for all $\beta\in (0, 1/2]$.
\end{proof}

\begin{rem}
   \label{rem:wsc} 
   Corollary      \ref{cor:cvZ}     gives      directly     that
   $(  |\tn|^{-3/2}\sum_{v\in  \tn} |\tnv|,  |\tn|^{-5/2}\sum_{v\in  \tn}
   |\tnv|^2)$
   is  asymptotically distributed  as  $\sqrt{2\alpha}  \, (Z_1,  Z_2)$.
Recall $\chi(\bt)$ defined in  \reff{eq:def-chi}. According     to    Lemma 3 of  \cite{bfj:mv}     or     \cite{f:pc},
there exists a finite constant $K$ such that,  for all $n\geq 3$,  $\E[\min(   |\tnu|, |\tnd|)]\leq
K|\tn|^{1/2}$. Since conditionally on $\{v\in \tn\}$ and $|\tnv|$, we have
that $\tnv$ is uniformly distributed on the trees with $|\tnv|$ nodes,
we deduce that $\E[\chi(\tn)]\leq  K
\E[\caa_{\tn}(\sqrt{x})]$. According to Theorem 3.8 in \cite{fk:ld}, we
have $\E[\caa_{\tn}(\sqrt{x})]= O(n\log(n))$ and thus
$\E[\chi(\tn)]=O(n \log(n))$. Noticing that $\chi(\tn)$ is
non-decreasing in $n$, and arguing as in Section
\ref{sec:proof-main}, we deduce that a.s. $\lim_{n\rightarrow+\infty }
|\tn|^{-3/2} \chi(\tn)=0$. 
Then, we   can  directly   recover  the   joint  asymptotic
   distribution of  the total length  path, the Wiener,  Sackin, Colless
   and  cophenetic indexes  defined by \reff{eq:dist-A1} in Section
   \ref{sec:measure} and \reff{eq:SCCo=A} 
   in Section 
   \ref{sec:index}   for the
   Catalan   model.    More   precisely,    we   have:
\[
\left(\frac{P(\tn)}{|\tn|^{3/2}}, \frac{W(\tn)}{|\tn|^{5/2}}, \frac{S(\tn)}{|\tn|^{3/2} },
  \frac{C(\tn)}{|\tn|^{3/2} }, 
\frac{\mathrm{Co}(\tn)}{|\tn|^{5/2}}\right)
\xrightarrow[n\rightarrow \infty ]{\text{a.s.}} 
\sqrt{2\alpha} \left(Z_1, 2 (Z_1 - Z_2),
     \frac{Z_1}{2},  \frac{Z_1}{2}, \frac{Z_2}{4}\right).
\] 
\end{rem}

\begin{rem}
  \label{rem:tildeF} 
We complete  Corollary \ref{cor:cvZ} by considering
  the additive  functionals $\tilde F$, see  definition \reff{eq:tildeF}
  used  in  \cite{fk:ld}, instead  $F$  defined  by \reff{eq:F=+b}.  For
  $\bt\in \T$ and $|\bt|>1$,  recall $\bt^*=\bt\setminus\cl(\bt)$ is the
  tree $\bt$ without its leaves.  We have that a.s.  for all $\beta> 0$,
\begin{equation}
   \label{eq:t*}
\lim_{n\rightarrow+\infty } |\tn^*|^{-(\beta+\frac{1}{2})}\sum_{v\in
  \tn^*}|\tnv^*|^{\beta}= 2\sqrt{ \alpha} \, Z_\beta.  
 \end{equation}
This result differs from Corollary \ref{cor:cvZ} as 
$\sqrt{2}$ is replaced by 2. To prove \reff{eq:t*}, first notice that for a full binary
tree $|\bt^*|=|\bt| - |\cl(\bt)|=(|\bt|-1)/2$ so that:
\[
 |\tn^*|^{-(\beta+\frac{1}{2})}\sum_{v\in
  \tn^*}|\tnv^*|^{\beta}
={\sqrt{2}}\,  (|\tn|-1)^{-(\beta+\frac{1}{2})}\sum_{v\in
  \tn}(|\tnv|-1)^{\beta}.
\]
Let  $x_+=\max(x,0)$ denote  the positive  part  of $x\in  \R$. We  have
$x^\beta\geq (x-1)^\beta\geq x^\beta -  c_\beta x^{(\beta-1)_+}$ for all
$x\geq 1$ with $c_\beta=1 $  if $0<\beta\leq 1$ and $c_\beta=\beta$ if
$\beta\geq 1$. Then use Corollary  \ref{cor:cvZ} (two times) to deduce
that a.s. for all $\beta>0$:
\[
\lim_{n\rightarrow+\infty } |\tn^*|^{-(\beta+\frac{1}{2})}\sum_{v\in
  \tn^*}|\tnv^*|^{\beta}= {\sqrt{2 }} \, 
\lim_{n\rightarrow+\infty } |\tn|^{-(\beta+\frac{1}{2})}\sum_{v\in
  \tn}|\tnv|^{\beta}= 2\sqrt{ \alpha} \, Z_\beta.  
\]
\end{rem}

The next proposition, whose proof is given in Section
\ref{sec:proof-fluc}, gives  the  fluctuations corresponding to   the 
invariance  principles  of  
Theorem \ref{theo:principal}. Notice  the speed of convergence
in the invariance principle is of order $|\tn|^{-1/4} $.

\begin{prop}\label{prop:fluctuation}
Let $f\in \cc([0,1])$ be locally Lipschitz  continuous on $(0,1]$ with
$\norm{x^a f'}_{\text{esssup}} $  finite for some $a\in 
(0,1)$. We have 
the following convergence in distribution:
 \[
\big(|\tn|^{1/4} (A_n- \sqrt{2 \alpha}\, \Phi_e)(f), \, A_n \big)
\; \xrightarrow[n\rightarrow \infty ]{(d)} \;
\left((2\alpha)^{1/4}
  \sqrt{\Phi_e(xf^2) }\,\,
  G, \, \sqrt{2 \alpha}\,\,  \Phi_e \right),
\]
where $G$ is a standard (centered reduced) Gaussian random variable independent of
the excursion $e$. 
\end{prop}

Notice  the fluctuations  for  the a.s.  convergence  towards $Z_\beta$  with
$\beta\geq  1$, given  in  Corollary \ref{cor:cvZ},  have an  asymptotic
variance (up  to a multiplicative  constant) given by  $Z_{2\beta}$.

\begin{rem}
   \label{rem:ext}
Using Remark \ref{rem:holder}, we can easily extend Proposition
\ref{prop:fluctuation} to  uniformly H\"older
continuous functions  $f$ on $[0,1]$ with H\"older exponent $\lambda>1/2$.
\end{rem}

\begin{rem}
   \label{rem:cv-Tn}
The contribution to the fluctuations  is given by the error of approximation
of $A_{n,1}(f)$  by $A_{n,2}(f)$,  see notations
from the proof of Theorem  \ref{theo:principal}. This corresponds to the
fluctuations  coming  from  the  approximation  of  the  branch  lengths
$(h_{n,v},  v\in  \tn)$  by  their   mean,  which  relies  on  the  explicit
representation on their joint  distribution given in Lemma \ref{lem:hh}.
In particular,  there is  no other contribution to  the fluctuations  from the
approximation  of the  continuum  tree $ \ct$  by the  sub-tree
$ \ct_{[n]}$.
\end{rem}

\subsection{Simply generated trees model}
\label{sec:main-simply}
We keep notations from Section \ref{sec:simply-tree} on simply generated
random tree. We  assume the weight sequence $\rp=(\rp (k),  k\in \N)$ 
        of        non-negative         real numbers        such        that
$\sum_{k\in  \N} \rp(k)>\rp(1)+\rp(0)$  and  $\rp(0)>0$  is generic.  As
stated in Section \ref{sec:simply-tree},  without loss of generality, we
will assume that $\rp$ is a critical probability distribution, that is:
\[
\sum_{k\in     \N}     \rp(k)=\sum_{k\in     \N}     k\rp(k)=1.
\]

The next result is a direct consequence of  \cite{d:ltcpcgwt} on the
convergence of the contour process of random discrete tree, see
Corollary \ref{cor:cvDk} given in Section \ref{sec:proof-cv-measure}. 
We keep notations and definitions of Section \ref{sec:proof-cv-measure},
with $H$ the normalized excursion of  the height function associated to the branching
mechanism $\psi$. 

\begin{cor}
   \label{cor:cv-measure}
   Let  $\rp$  be a  critical  probability  distribution on  $\N$,  with
   $1>\rp(1)+\rp(0)\geq  \rp(0)>0$,  which  belongs  to  the  domain  of
   attraction  of a  symmetric stable  distribution of  Laplace exponent
   $\psi(\lambda)= \kappa  \lambda ^\gamma$  with $\gamma\in  (1,2]$ and
   $\kappa>0$,  and  renormalizing  sequence $(a_p,  p\in  \N^*)$.   Let
   $\tau$  be  a   GW  tree  with  offspring   distribution  $\rp$,  and
   $\tau^{(p)}$    be   distributed    as   $\tau$    conditionally   on
   $\{|\tau|=p\}$. We have the following convergence in distribution:
\[
\frac{a_p}{p^2} \caa_{\tau^{(p)}}
\xrightarrow[p\rightarrow+\infty]{(d) } \Phi_H,
\]
where we endow the space of non-negative measures with the topology of
the weak converge and where the convergence is taken  along the
infinite  sub-sequence of  $p$  such that  $\P(|\tau|=p)>0$. 
\end{cor}

We set for $\beta>0$ and $\bt\in \T$:
\[
Z_\beta^*(\bt)=\sum_{v\in \bt}|\bt_v|^{\beta}.
\]

\begin{cor}\label{cor:GW}
  Under the hypothesis and notations of Corollary \ref{cor:cv-measure}, we
  have the following convergence in distribution  
for all $\beta\geq 1$,
\begin{equation}
   \label{eq:cvZH}
\frac{a_p}{p^{\beta+1} } Z^{*}_\beta (\tau^{(p)})
\xrightarrow[p\rightarrow+\infty]{(d) } Z_\beta^H,
\end{equation}
with $Z_\beta^H$ given by \reff{eq:zb} and where the convergence is taken  along the
infinite  sub-sequence of  $p$  such that  $\P(|\tau|=p)>0$. 
\end{cor}

The proof of the first part of the next Lemma is given in Section
\ref{sec:H}. The second part, which is the representation formula,  is a
direct consequence of the deterministic Lemma \ref{lem:sigma-moment} in Section \ref{sec:proofZ} (with $\beta=a+1$). 

\begin{lem}
   \label{lem:ZH}
Assume the height function $H$ is associated to the Laplace exponent $\psi(\lambda)=\kappa \lambda^\gamma$ with $\gamma\in
(1,2]$ and $\kappa>0$.  We have that
a.s.  for all $1/\gamma\geq \beta>0$,   $Z_\beta^H=+\infty
$, that  a.s.  for all $\beta>1/\gamma$,  $Z_\beta^H$ is finite and 
\begin{equation}
   \label{eq:EZH}
\E\left[Z_\beta^H\right]=\inv{\gamma \kappa^{1/\gamma}}\, \frac 
{\Gamma\left(\beta-\inv{\gamma}\right)}{\Gamma\left(\beta+1-\frac{2}{\gamma}\right)}\cdot
\end{equation}
We   also   have   the
representation formulas $Z_1^H= \int_{0}^{1}H(s)\, ds$ and, for $\beta>1$,
$Z_\beta^H=  \inv{2} \beta  (\beta-1)\int_{[0,1]^2} |t-s|^{\beta  -2} \,
m_H(s,t)                                                      \,ds\,dt$.\\
\end{lem}

\begin{rem}
   \label{rem:conj}
We conjecture that  for simply generated trees, under the
hypothesis  of    Corollary  \ref{cor:cv-measure}, there  is  a  phase
transition   at   $\beta=1/\gamma$   between   a   ``global''   regime
($\beta>1/\gamma$), where the convergence \reff{eq:cvZH} holds, and a ``local'' regime
($\beta\leq 1/\gamma$) where the convergence \reff{eq:cvZH} is
irrelevant as the right-hand-side of \reff{eq:cvZH}  is a.s. infinite.
\end{rem}

Using the Skorohod representation theorem, notice that
all the  convergences in  distribution of Corollary
\ref{cor:GW} hold
simultaneously.

\begin{rem}
   \label{rem:s2-finit}
   If  $\rp$ has  finite variance,  say  $\sigma^2$, then  one can  take
   $a_p=\sqrt{p}$     in     Corollaries    \ref{cor:cv-measure}     and
   \ref{cor:GW}  and  $H$  is  equal  to  $(2/\sigma)  B$  which
   corresponds to $\psi(\lambda)= \sigma^2\lambda  ^2/2$, see Remarks
   \ref{rem:s2-fini-0} and \ref{rem:s2-fini-1}. By scaling, or
   using that the limit in  Theorem \ref{theo:principal} does not depend
   on     $\alpha$,     we     deduce      that     in     this     case
   $\Phi_{H}=\frac{2}{\sigma}        \Phi_B$     and $Z^H_\beta=
   \frac{2}{\sigma}     Z_\beta^B$  in        Corollaries
   \ref{cor:cv-measure} and \ref{cor:GW}.
\end{rem}

\section{Preliminary Lemmas}
\label{sec:premlim-lem}
% We state in this section Lemmas, which will be used in Section
% \ref{sec:proof-main} to prove Theorem \ref{theo:principal}.

Recall $\ct$  is the real tree  coded by the excursion  $e$, see Section
\ref{sec:Te}  and $\ct_{[n]}  $ is  the (smallest)  sub-tree of  $\ct_e$
containing $n+1$  leaves picked  uniformly at random  and the  root, see
Section \ref{sec:T[n]}.  Recall $(\tn, (h_{n,v}, v\in  \tn))$ denote the
corresponding marked tree. Intuitively, for $v\in \tn$, $h_{n,v}$ is the
length  of the  branch  below  the branching  point  with  label $v$  in
$\ct_{[n]}$  (when keeping  the order  on the  leaves).  We  recall, see
\cite{a:III}, \cite{p:csp}  (Theorem 7.9) or \cite{dlg:rtlpsbp},  that the
density of $(h_{n,v}, v\in \tn)$ is, conditionally on $\tn$, given by:
\begin{equation}
   \label{eq:dens-h}
 f_{n}((h_ {n,v},v\in \tn))=2\frac{(2n)!}{n!}\, \alpha^{n+1}\, L_n\, 
 \expp{-\alpha L_n^{2}}\prod_{v\in \tn} \ind_{\{h_{n,v}>0\}},
 \end{equation} 
 where  $L_n=\sum_{v\in   \tn}h_{n,v}$  denotes  the  total   length  of
 $  \ct_{[n]}$.   Notice  that  the edge-lengths  have  an  exchangeable
 distribution and are independent of the shape tree $\tn$.  Furthermore,
 elementary computations  give that $(h_{n,v},v\in \tn)$, with $v\in
 \tn$ ranked in the lexicographic order,  has, conditionally on $\tn$ and
 $L_n$, the same
 distribution  as $(L_n  \Delta_1,  \ldots,  L_n \Delta_{2n+1})$,  where
 $\Delta_1, \ldots, \Delta_{2n+1}$ represents  the lengths of the $2n+1$
 intervals    obtained   by    cutting    $[0,1]$    at   $2n$   independent
 uniform random variables on $[0,1]$  and independent of $L_n$.  We thus
 deduce the following elementary Lemma.

\begin{lem}
   \label{lem:hh}
Conditionally on $\tn=\bt$, the random vector $(h_{n,v},v\in \bt)$ has the same
distribution as $\left(L_nE_v/S_\bt,v\in \bt\right)$,
where $(E_u, u\in \cu)$  are   independent exponential 
random variables with mean 1, independent of $\tn$ and $L_n$, and
$S_{\bt}=\sum_{v\in \bt}E_v$. 
\end{lem}

According      to     \cite{ad:13},      we      have     that      a.s.
$\lim_{n\rightarrow+\infty  }   L_n/\sqrt{n}=1/\sqrt{\alpha}$.  We  then
deduce          from           Lemma          \ref{lem:hh}          that
$(2n+1)\sqrt{\alpha}\,     h_{n,\emptyset}/\sqrt{n}$     converges     in
distribution    towards    $    E_\emptyset$     as    $n$    goes    to
infinity.           Intuitively,          we           get          that
$2\sqrt{\alpha n}\, \E[h_{n,\emptyset}]$ is of  order 1, for $v\in \tn$.
Recall  the random  measure $A_n$  is defined  in \reff{eq:def-An}.   We
introduce the random measure:
\[
A_{1,n}= 2\sqrt{\alpha n}\,\E[h_{n,\emptyset}]A_{n}.
\]

\begin{lem}
\label{lem:A1}
  Let $a\in [0, 1/2)$. There exists  a finite constant $C$ such that for
  all $f\in\mathcal{B}([0,1])$ and $n\in\N^*$, we have:
\[
\E\left[\lvert A_{n}(f)-A_{1,n}(f)\rvert\right] \le C\lVert
 x^{a}f\rVert_{\infty}\;n^{-1}.
\] 
\end{lem}

\begin{proof}
Let $a\in [0, 1/2)$ and $f\in\cb([0,1])$.
Using  \reff{mh1} in the Appendix, we deduce that for all $n\in\N^{*}$, we have
$\left| 1-2\sqrt{\alpha n}\;\E[h_{n,\emptyset}]\right|\le 1/2n$. Using
\reff{mom_1} in Lemma \ref{lem_mom_1}, we deduce that:
\[
\E[\lvert A_{n}(f)-A_{1,n}(f)\rvert] 
\le \frac{1}{2n}\E[|A_{n}(f)|]
\le \frac{ C_{1,1-a}}{2n}\,  \lVert x^{a}f\rVert_{\infty}.
\]
\end{proof}
Intuitively, $h_{n,v}$  is of the  same order of its  expectation. Since
the random variables  $(h_{n,v}, v\in \tn)$ are  exchangeable, we deduce
that $h_{n,v}$ is of the  same order as $\E[h_{n,\emptyset}]$.  Based on
this intuition, we  define the random measure $A_{2,n}$  as follows. For
$f\in \cb([0, 1])$, we set:
\[
A_{2,n}(f)=2\sqrt{\alpha n}\,\,  |\tn|^{-3/2}\sum_{v\in
  \tn }|\tnv|f\left(\frac{| \tnv |}{|\tn|}\right) \, h_{n,v}.
\]

\begin{lem}
\label{lem:A2}
Let $a\in[0, 1/2)$. There exists a finite constant $C$ such that for all
$f\in\cb([0,1])$ and  $n\in\N^*$, we have:
\[
\E[\left|A_{1,n}(f)-A_{2,n}(f)\right|]
\le C \lVert x^{a}f\rVert_{\infty}\;n^{-1/4}.
\]
\end{lem}

\begin{proof}
Let $a\in [0, 1/2)$ and $f\in\cb([0,1])$.
For $v\in \tn$, we set $Y_{n,v}=\sqrt{n} (\E[h_{n,v}]-h_{n,v})$ and
\[
K_{n}= \inv{2\sqrt{\alpha}} (A_{1,n}(f)-A_{2,n}(f))
=|\tn|^{-3/2}\sum_{v\in
  \tn}|\tnv|f\left(\frac{|\tnv|}{|\tn|}\right)Y_{n,v}.
\]
Using that $(h_{n,v}, v\in \tn)$ is exchangeable,
elementary computations give:
\[
\E\left[K_n^2 |\tn\right]
\leq  |\tn|^{-1/2} A_n(xf^2) \E[Y_{n,\emptyset}^2] + A_n(|f|)^2
|\E[Y_{n,\emptyset} Y_{n,1}] |.
\]
Then using \reff{mom_1} and \reff{mom_3} in Lemma \ref{lem_mom_1} and
\reff{mh4} in Lemma \ref{moment_hauteurs}, we get: 
\[
 \E[K_{n}^{2}]=\E\left[\E[K_{n}^{2}|\tn]\right]
\le \frac{C_{1,1}}{2\alpha \sqrt{2n+1}}  \lVert x^{1/2}f\rVert^ 2_{\infty} 
+\frac{C_{2, 1-a}^2}{8\alpha n}  \lVert
x^{a}f\rVert_{\infty}^{2}
\le \frac{c}{\sqrt{n}} \lVert x^{a}f\rVert_{\infty}^{2},
\]
for some finite constant $c$ which does not depend on $n$ and $f$. 
\end{proof}

Let $\cl_{n,v}=\{u\in \tn; \, v\preccurlyeq  u,\, k_u(\tn)=0\}$ be the
set of 
leaves of $\tn$ 
with ancestor $v$, and $|\cl_{n,v}|$ be its cardinal. Notice the number
of leaves of $\tnv$ is exactly $|\cl_{n,v}|$. 
We now approximate the multiplying factor $|\tnv|$ in $A_{2,n}$ by twice
the number of leaves in $\tnv$ as $2|\cl_{n,v}|=|\tnv|+1$. For this reason, we
set for $f\in \cb([0,1])$:
\[
A_{3,n}(f)= 4\sqrt{\alpha n}\,\,   |\tn|^{-3/2}\sum_{v\in
  \tn } |\cl_{n,v}|\, f\left(\frac{| \tnv |}{|\tn|}\right) \, h_{n,v}.
\]

\begin{lem}
\label{lem:A3}
Let $a\in[0, 1/2)$. For all
$f\in\mathcal{B}([0,1])$ and  $n\in\N^*$, we have:
\[
\E[\left|A_{2,n}(f)-A_{3,n}(f)\right|]
\le  \lVert x^{a}f\rVert_{\infty}\;n^{a-\inv{2}}.
\]
\end{lem}

\begin{proof} 
Let $a\in [0, 1/2)$ and $f\in\cb([0,1])$.
Since $2|\cl_{n,v}|=|\tnv|+1$, we get that:
\[
\left|A_{2,n}(f)-A_{3,n}(f)\right|
\le 2\sqrt{\alpha n}\,|\tn|^{-3/2}\sum_{v\in
  \tn } |f|\left(\frac{| \tnv |}{|\tn|}\right) \, h_{n,v}.
\]
As $|\tnv|\geq 1$ and $a\geq 0$, we get that $|f|\left(\frac{| \tnv
    |}{|\tn|}\right) \leq  \norm{x^a f}_\infty |\tn|^{a}$. We deduce
that:
\[
\left|A_{2,n}(f)-A_{3,n}(f)\right|
\le 2\sqrt{\alpha n}\,L_n |\tn|^{a -\frac{3}{2}} \norm{x^af}_\infty.
\]
According to \reff{eq:moments_Ln2}, we have $2\sqrt{\alpha n}\,
\E[L_n]\leq  |\tn|$. We deduce that $
\E[\left|A_{2,n}(f)-A_{3,n}(f)\right|]
\leq  |\tn|^{a-\inv{2}} \norm{x^af}_\infty$. 
\end{proof}

We define  $\cnrk$ as the number  of leaves of the  sub-tree $\ct_{[n]}$
which are distinct from ${\bf{p}}(U_k)$  and such that their most recent
common ancestor  with ${\bf{p}}(U_k)$  is at  distance further  than $r$
from the root.  More precisely,  using the definition \reff{eq:def-m} of
$m$, we have:
\[
\cnrk+1=\Card\{ i\in \{1, \ldots, n+1\}, \, m(U_i, U_k)\geq r\}.
\]
In particular, we deduce from the construction of $\ct_{[n]}$ and $\tn$
that for $1\leq k\leq n+1$:
\begin{equation}
   \label{eq:s=i}
  \sum_{v\preccurlyeq u(U_k) }  f\left(\frac{| \tnv |}{|\tn|}\right) \, h_{n,v}
=\int_0^{e(U_k)}  dr\,  f\left(\frac{2\cnrk+1}{2n+1}\right),
\end{equation}
where $u(U_k)$ is the leaf in $\tn$ corresponding to the leaf 
${\bf{p}}(U_k)$ in $\ct_{[n]}$.

Recall that, for $v\in \tn$,   $\cl_{n,v}$ denotes the set of leaves of
$\tn$ with ancestor $v$ and $\cl(\tn)=\cl_{n,\emptyset}$ denotes the set
of leaves of $\tn$.
We deduce that:
\begin{align*}
A_{3,n}(f)
&= 4\sqrt{\alpha n}\,\,   |\tn|^{-3/2}\sum_{v\in
  \tn } |\cl_{n,v}|\, f\left(\frac{| \tnv |}{|\tn|}\right) \, h_{n,v}\\
&= 4\sqrt{\alpha n}\,\,   |\tn|^{-3/2}\sum_{u\in \cl(\tn)}
  \sum_{v\preccurlyeq u }  f\left(\frac{| \tnv |}{|\tn|}\right) \, h_{n,v}\\
&= 4\sqrt{\alpha n}\,\,   |\tn|^{-3/2}\sum_{k=1}^{n+1}
\int_0^{e(U_k)}  dr\,  f\left(\frac{2\cnrk+1}{2n+1}\right),
\end{align*}
where  we used  \reff{eq:s=i} for  the  last equality.   Notice that  by
construction,  conditionally  on  $e$  and $U_k$,  the  random  variable
$\cnrk$  is  binomial with  parameter  $(n,  \sigma_{r,U_k})$. For  this
reason,  we consider  the following  approximation of  $A_{3,n}(f)$.  For
$f\in \cb([0,1])$ non-negative, we set:
\[
A_{4,n}(f) = 
 4\sqrt{\alpha n}\,\,   |\tn|^{-3/2}\sum_{k=1}^{n+1}
\int_0^{e(U_k)}  dr\,  f(\sigma_{r,U_k}).
\]

\begin{lem}
\label{lem:A4}
We have the following properties.
\begin{enumerate}
\item[(i)] For $a\in (0,1)$, there exists a finite constant $C(a)$ such that
  if $f\in \cb([0,1])$ is locally Lipschitz  continuous on $(0,1]$,    we have for all $n\in\N^*$:
\[
\E[\left|A_{3,n}(f)-A_{4,n}(f)\right|]
\le C(a)\norm{ x^{a}f^{'}}_{\text{esssup}}\,n^{-1/2}.
\]
\item[(ii)] If $a\in (-1/2,0]$, there exists a finite constant $C(a)$ such that  we have for all $n\in\N^*$:
\[
\E[\left|A_{3,n}(x^a)-A_{4,n}(x^a)\right|]
\le C(a)\, n^{-(2a+1)/8}.
\]
% \item 
% There exists a finite constant $C$ such that for all ${\lambda}>1/2$ and all
% $f$ uniformly H\"older continuous with exponent ${\lambda}$,  we have for all $n\in\N^*$:
% \[
% \E[\left|A_{3,n}(f)-A_{4,n}(f)\right|]
% \le C \lVert f\rVert_{C^{0,{\lambda}}}\;n^{-\lambda/2},
% \]
% where $\lVert f\rVert_{C^{0,{\lambda}}}=\sup_{0\leq x<y\leq 1}\frac{|f(x)-f(y)|}{|x-y|^{\lambda}}$ ;
\end{enumerate}
\end{lem}
\begin{rem}
   \label{rem:holder}
We can extend (i) of Lemma \ref{lem:A4} to get that for uniformly H\"older
continuous function $f$ with exponent $\lambda>1/2$, we have
$\E[\left|A_{3,n}(f)-A_{4,n}(f)\right|]=O(n^{-\lambda/2})$. This allows to
extend Proposition \ref{prop:fluctuation} to such functions. 
\end{rem}

\begin{proof}
  For  $s\in  [0,1]$,  let  $\cnrs$  be  a  random  variable  which  is,
  conditionally    on     $e$,    binomial    with     parameter    $(n,
  \sigma_{r,s})$.
  Notice, this is  consistent with the definition of  $\cnrk$.  Hence we
  get, for $f\in \mathcal{B}([0,1])$,
\begin{align}
\nonumber \E\left[\left|A_{3,n}(f)-A_{4,n}(f)\right|\right]
 &\le 4\sqrt{\alpha
   n}|\tn|^{-\frac{3}{2}}\sum_{k=1}^{n+1}\E\left[\int_{0}^{e(U_{k})}
   \left| f\left(\frac{2\cnrk+1}{2n+1}\right) - f(\sigma_{r,U_k})\right|dr\right] \\ 
 \label{eq:A345}
&\le 4 \sqrt{\alpha}\int_{0}^{1}ds\;\E\left[\int_{0}^{e(s)}dr \, \E\left[\left|
  f\left(\frac{2\cnrs+1}{2n+1}\right)-f(\sigma_{r,s})\right|\,\big|\, e\right]\,dr\right]. 
 \end{align}
\\

We first prove  property (i). Let $a\in (0,1)$ and  $f\in \cb([0,1])$ be
locally    Lipschitz   continuous    on   $(0,1]$.    Using   (ii)    of
Lemma~\ref{lem_binomial},   we   have   that  for   $s\in   (0,1)$   and
$r\in (0,e(s))$,
\begin{equation}
   \label{eq:A34}
\E\left[\left|
  f\left(\frac{2\cnrs+1}{2n+1}\right)-f(\sigma_{r,s})\right|\,\big|\,
e\right]
 \leq \frac{\norm{ x^{a}f'}_{\text{esssup}}}{1-a}\, 
\left(\sigma_{r,s}^{-\frac{a}{2}}
+ \sigma_{r,s}^{\frac{1}{2}-a}\right)n ^ {-1/2}.
\end{equation}
We    recall    that
$Z_{\beta}=\int_{0}^{1}ds\int_{0}^{e(s)}dr\;\sigma_{r,s}^{\beta-1}$ for
$\beta>0$.            Thus, we have
$\E\left[Z_{\frac{3}{2}-a}\right]\leq \E\left[Z_{1-\frac{a}{2}}\right]$;
the last term being finite thanks to  Lemma~\ref{lem:subexcursion}.
 We deduce from
\reff{eq:A345} and \reff{eq:A34} that
\[
\E\left[\left|A_{3,n}(f)-A_{4,n}(f)\right|\right]\leq  8
\sqrt{\alpha}\frac{\norm{x^{a}f'}_{\text{esssup}}}{1-a}\,
\E\left[Z_{1-\frac{a}{2}}\right] n ^ {-1/2}.
\]
This achieves the proof of property (i).\\

We now prove property (ii). We consider $a\in (-1/2, 0)$ and $f(x)=x^a$,
as the case $a=0$ is obvious. Let ${\gamma}>0$. We write:
\[
\int_{0}^{1}ds\;\E\left[\int_{0}^{e(s)} dr\, \E\left[\left|
  \left(\frac{2\cnrs+1}{2n+1}\right)^a-\sigma_{r,s}^a\right|\,\big|\,
e\right]\right] =\kappa_1+\kappa_2+\kappa_3,
\]
with $\kappa_i=\int_{0}^{1}ds\;\E\left[\int_{0}^{e(s)}dr\, \E\left[\ind_{D_i} \left|
  \left(\frac{2\cnrs+1}{2n+1}\right)^a-\sigma_{r,s}^a\right|\,\big|\,
e\right]\right]$ and:
\[
D_1=\left\{\sigma_{r,s}>2n^{-{\gamma}},\frac{2\cnrs+1}{2n+1}>n^{-{\gamma}}\right\}, 
\quad
D_2=\left\{\sigma_{r,s}>2n^{-{\gamma}},\frac{2\cnrs+1}{2n+1}\leq
  n^{-{\gamma}}\right\}, 
\]
and $D_3=(D_1\bigcup  D_2)^c$. For $\kappa_1$, we have:
\begin{equation}
\label{eq:1-lem:A4}
\kappa_1
 \leq
{n^{{\gamma}(1-a)}}\int_{0}^{1}\!ds\,\E\left[\int_{0}^{e(s)}\E\left[\left|
 \frac{2\cnrs+1}{2n+1}-\sigma_{r,s}\right|\,\big|\,
e\right]dr\right] 
 \leq {\E \left[\int_0^1  \!\!e(s)\, ds\right]} \, n^{{\gamma}(1-a)-\inv{2}},
\end{equation}
where   we   used,   thanks   to  \reff{ineq_3}   with   $b=1+a$,   that
$|x^a  -  y^a|\leq  x^a  y^{-1} |x-y|  \leq  n^{{\gamma}(1-a)}|x-y|$  if
$x,y\in [n ^ {-\gamma}, +\infty )$  for the first inequality and for the
second that if $X$ is a binomial random variable with parameter $(n,p)$,
then we have:
\[
\E\left[\left|\frac{2X+1}{2n+1} - p\right|\right]^2\leq 
\E\left[\left(\frac{2X+1}{2n+1} - p\right)^2\right]\leq
\inv{2n+1}\leq  \inv{n}\cdot
\]

We give an upper bound of 
 $\kappa_2$.  We first recall Hoeffding's inequality: if $X$ is a binomial random variable with 
parameter $(n,p)$, and $t>0$, then we have  ${\mathbb P}(np-X>nt)\leq
\exp(-2nt^2)$. Using that $\{p -\frac{2X+1}{2n+1}> n^{-\gamma} \}\subset \{np- X >
n^{1-\gamma}\}$, we deduce that:
\begin{equation}\label{eq:Hoeffding}
{\mathbb P}\left(p- \frac{2X+1}{2n+1}>n^{-\gamma}\right)
\leq
{\mathbb P}\left(np-{X} >n^{1-\gamma}\right)
\leq \exp\left(-2 n^{1-2\gamma}\right).
\end{equation}

Notice that on $D_2$, we have $ 0\leq
\left(\frac{2\cnrs+1}{2n+1}\right)^a-\sigma_{r,s}^a \leq
\left(\frac{2\cnrs+1}{2n+1}\right)^a\leq (2n+1)^{-\gamma a}$ as well as
$\sigma_{r,s} - \frac{2\cnrs+1}{2n+1}> n^{-\gamma}$. 
Hence, we obtain:
\begin{align}
\nonumber
\kappa_2
&\leq
(2  n+1)^{-{\gamma}a}\int_{0}^{1}ds\,\E\left[\int_{0}^{e(s)}
   \P\left(\sigma_{r,s}  
  - \frac{2\cnrs+1}{2n+1}>n^{-{\gamma}} \Big|
  e\right)\,dr\right]\\
&\leq\E[Z_1] \, (2  n+1)^{-{\gamma}a}\expp{-2{n^{1-2\gamma}}}.
\label{eq:2-lem:A4}
\end{align} 

Finally, we consider  $\kappa_3$. 
Let $\eta\in (0,a+1/2)$. We have:
\begin{multline*}
 \E\left[\int_{0}^{e(s)}\ind_{\{\sigma_{r,s}\leq
    2n^{-{\gamma}} \}} 
\E\left[\left|
  \left(\frac{2\cnrs+1}{2n+1}\right)^a-\sigma_{r,s}^a\right|\,\big|\,
e\right]\,dr\right]\\  
\begin{aligned}
&\leq  
\E\left[\int_{0}^{e(s)}\ind_{\{\sigma_{r,s}\leq
    2n^{-{\gamma}} \}} 
\E\left[
  \left(\frac{2\cnrs+1}{2n+1}\right)^a+\sigma_{r,s}^a\,\big|\,
e\right]\,dr\right]\\
&\leq  
3 \E\left[\int_{0}^{e(s)}\ind_{\{\sigma_{r,s}\leq
    2n^{-{\gamma}} \}} \sigma_{r,s}^a \,dr\right]\\
&\leq  3\cdot 2^{\eta}n^{-\gamma\eta}\;\E\left[\int_{0}^{e(s)}dr\,\sigma_{r,s}^{a-\eta}\right],
\end{aligned}
\end{multline*}
where we used  (i) of Lemma~\ref{lem_binomial} for the second inequality. 
Recall  that $D_3=
\left\{\sigma_{r,s}\leq 2n^{-{\gamma}} \right\}$. We deduce that:
\begin{equation}
\label{eq:3-lem:A4}
\kappa_3\leq  \int_0^1 ds\, 3\cdot
2^{\eta}n^{-\gamma\eta}\;\E\left[\int_{0}^{e(s)}dr\,\sigma_{r,s}^{a-\eta}\right]
= 3\cdot
2^{\eta}n^{-\gamma\eta}\; \E[Z_{a-\eta+1}].
\end{equation}

Choose    $\gamma=1/3$    and    $\eta=3(2a+1)/8$.     Thanks to Lemma
\ref{lem:subexcursion}, we   get  that $\E \left[\int_0^1  e(s)\,
  ds\right]=\E[Z_1]$ is finite and that 
$\E[Z_{a-\eta+1}]$  is also finite
since   $a-\eta+1>1/2$  as   $a>  -1/2$.   Therefore,  we   deduce  from
\reff{eq:A345} and then 
\reff{eq:1-lem:A4}, \reff{eq:2-lem:A4} and \reff{eq:3-lem:A4} that there exists a finite constant $C(a)$ such
that  we have for all $n\in\N^*$: 
\[
\E[\left|A_{3,n}(x^a)-A_{4,n}(x^a)\right|]
\le C(a)\, n^{-(2a+1)/8}.
\]
\end{proof}

% 3. Let ${\lambda}>1/2$ and  $f$ uniformly H\"older continuous with
% exponent ${\lambda}$ on $[0,1]$. 

% \begin{align*}
%  \E\left[\left|A_{3,n}(f)-A_{4,n}(f)\right|\right]
% &\le\sqrt{\alpha}\int_{0}^{1}ds\;\E\left[\int_{0}^{e(s)}\E\big[| f\left(\frac{2\cnrs+1}{2n+1}\right)-f(\sigma_{r,s})|\big|e\big]\,dr\right]\\
% &\le\sqrt{\alpha}\lVert f\rVert_{C^{0,{\lambda}}}\int_{0}^{1}ds\;\E\left[\int_{0}^{e(s)}\E\big[| \frac{2\cnrs+1}{2n+1}-\sigma_{r,s}|^{\lambda}\big|e\big]\,dr\right]\\
% &\le\sqrt{\alpha}\lVert f\rVert_{C^{0,{\lambda}}}\int_{0}^{1}ds\;\E\left[\int_{0}^{e(s)}\E\big[| \frac{2\cnrs+1}{2n+1}-\sigma_{r,s}|^{\lambda}\big|e\big]\,dr\right].
%  \end{align*}
% Using that $| \frac{2\cnrs+1}{2n+1}-\sigma_{r,s}|\leq 1$, we have that
% the last quantity is a non-increasing function of ${\lambda}$. We may
% then assume that ${\lambda}\in (1/2,2]$. Using that if $X$ is a
% binomial random variable with parameter $(n,p)$, ${\text{Var}}(X)\leq
% \sqrt{pn}$, and   H\"older inequality, we have that for some constant
% $C$ which does not depend on $n$, ${\lambda}$ and $f$, 

% \begin{align*}
%  \E\left[\left|A_{3,n}(f)-A_{4,n}(f)\right|\right]
% \le C\lVert f\rVert_{C^{0,{\lambda}}}n^{-{\lambda}/2}\int_{0}^{1}ds\;\E\left[\int_{0}^{e(s)}\sigma_{r,s}^{\lambda/2}\,dr\right]
% \leq C\lVert f\rVert_{C^{0,{\lambda}}}n^{-1/4}\cdot\\
%  \end{align*}

\begin{lem}
\label{lem:LGN}
For all $f\in\mathcal{B}([0,1])$ such that $f\geq 0$ and $\lVert
x^{a}f\rVert_{\infty}<+\infty$ for some $a \in [0, 1/2)$, we have:
 \[
A_{4,n}(f) 
\,\xrightarrow[n\rightarrow+\infty]{a.s.} \,
\sqrt{2\alpha} \, \, \Phi_e(f). 
\]
\end{lem}

\begin{proof}
Let $f\in\mathcal{B}([0,1])$ such that $f\geq 0$ and $\lVert
x^{a}f\rVert_{\infty}<+\infty$ for some $a \in [0, 1/2)$. Let $U$ be
uniform on $[0,1]$ and independent of $e$. Recall $Z_\beta=\int_0^1 ds \int_0^1 dr
\, \sigma_{r,s}^{\beta-1}$ defined in
\reff{eq:zb}. 
Notice that:
\[
\E\Big[\int_{0}^{e(U)}dr\, f(\sigma_{r,U})\,\Big|\, e\Big]
\leq  \norm{x^a f}_\infty Z_{1-a}.
\]
Since $1-a>1/2$, we deduce from Lemma \ref{lem:subexcursion} that a.s.
$Z_{1-a}<+\infty $.
Then, use the strong law of large numbers (conditionally on $e$) to
deduce that $A_{4,n}(f)$ converges a.s. towards  $\sqrt{2\alpha} \, \,
\Phi_e(f)$ as $n$ goes to infinity. 
\end{proof}

\section{Proof of Theorem \ref{theo:principal}}
\label{sec:proof-main}

Let $a>-1/2$. According to Lemmas
\ref{lem:A1}, \ref{lem:A2}, \ref{lem:A3} and \ref{lem:A4} (use (i) for
$a>0$ and (ii) for $a\in (-1/2,0]$), there exists
$\varepsilon>0$ and a finite constant $c$  such that for all $n\in
\N^*$, we have $\E[|A_n(x^a) - A_{4,n}(x^a)|]\leq  c
n^{-\varepsilon}$. Since according to Lemma \ref{lem:LGN},  we have
a.s. that $\lim_{n\rightarrow+\infty } A_{4,n}(x^a) =
\sqrt{2\alpha} \, \, \Phi_e(x^a)$, we deduce from Borel-Cantelli lemma
that,  with $\varphi(n)=\lceil
n^{2/\varepsilon} \rceil$, 
we have  a.s. $\lim_{n\rightarrow+\infty } A_{\varphi(n)}(x^a) =
\sqrt{2\alpha} \, \, \Phi_e(x^a)$. 

For $n'\geq n\geq 1$, we have $\ct_{[n]}\subset
\ct_{[n']}$. Unfortunately, by the construction of $\tn$, we don't have
in general that $v\in \tn$ implies that $v\in {\rm T}_{n'}$. However, it
is still true, as $1+a>0$, that:
\begin{equation}
   \label{eq:ineq-tn}
\sum_{v\in {\rm T}_{n}}  |{\rm T}_{n, v} |^{1+a}
\leq 
\sum_{v'\in {\rm T}_{n'}}  |{\rm T}_{n', v'} |^{1+a}.
\end{equation}

Let $n\in \N^*$. There exists a unique $n'\in \N^*$ such that
$\varphi(n')\leq n<\varphi(n'+1)$. We obtain from \reff{eq:ineq-tn} that:
\[
\left(\frac{2 \varphi (n') +1}{2 \varphi (n'+1) +1}\right)^{a  +\frac{3}{2}} 
  A_{\varphi(n')}(x^a) \leq  A_n(x^a) \leq  
\left(\frac{2 \varphi (n'+1)+1 }{2 \varphi (n') +1}\right)^{a  +\frac{3}{2}} 
A_{\varphi(n'+1)}(x^a) .
\]
As $\lim_{n'\rightarrow+\infty } \varphi(n')/\varphi(n'+1)=1$, we deduce
that a.s. $\lim_{n\rightarrow+\infty } A_n(x^a)= \sqrt{2\alpha} \, \,
\Phi_e(x^a)$. \\

In particular, for all $a\in (-1/2,0]$,  a.s. for all $k\in \N$, we have
$\lim_{n\rightarrow+\infty   }   A_n(x^{a+k})=  \sqrt{2\alpha}   \,   \,
\Phi_e(x^{a+k})$.
Since  on  $[0,1]$,   the  convergence  of  moments   implies  the  weak
convergence  of  measure,  we  deduce  that  a.s.   the  random  measure
$A_n(x^a\,      \bullet)      $       converges      weakly      towards
$\sqrt{2\alpha} \, \,  \Phi_e(x^{a}\, \bullet)$. By taking  a dense subset
of $a$ in $(-1/2, 0]$ and using monotonicity, we deduce  that a.s. for
all $a\in (-1/2, 0]$ the 
random   measure  $A_n(x^a\,   \bullet)  $   converges  weakly   towards
$\sqrt{2\alpha} \, \, \Phi_e(x^{a}\, \bullet)$. This ends the proof of
Theorem \ref{theo:principal}.

\section{Proof of Proposition \ref{prop:fluctuation}}
\label{sec:proof-fluc}

\subsection{A preliminary stable convergence}

Let $(E_v, v\in \cu)$ be independent exponential random variables with
mean 1 and independent of $e$. Let $f\in \cc([0,1])$. We set for $v\in \tn$:
\begin{equation}
   \label{eq:def-Zn}
X_{n,v}=  |\tn|^{-5/4} |\tnv| f \left(\frac{|\tnv|}{|\tn|}\right)
\quad\text{and}\quad
Z_n(f) =\sum_{v\in \tn} X_{n,v}\,  (E_v -1).
\end{equation}
We have the following lemma. 
\begin{lem}
   \label{lem:cv-ZA}
Let $f\in \cc([0,1])$ be locally Lipschitz  continuous on $(0,1]$ such
that $\norm{x^a f'}_{\text{esssup}}$ is finite for some $a\in (0,1)$.  We have
the following stable convergence: 
\begin{equation}
   \label{eq:cv-ZA}
(Z_n(f), A_n)
\,\xrightarrow[n\rightarrow+\infty]{(d) } \,
\left((2 \alpha)^{1/4 }\sqrt{ \Phi_e(xf^2)} \,\, G, \sqrt{2 \alpha}\, \, \Phi_e\right),
\end{equation}
where $G$ is a standard   Gaussian random variable independent of $e$. 
\end{lem}

\begin{proof}
  Let $f\in \cc([0,1])$.  We first  assume that $f$ is non-negative.  We
  compute the Laplace transform of $Z_n(f)$ conditionally on $\tn$.  Let
  $\lambda>0$. Elementary computations give:
\[
   \E\left[\expp{-\lambda Z_n(f)}|\tn\right]
=\expp{\lambda \sum_{v\in \tn} X_{n,v}} \E\left[\expp{-\lambda
  \sum_{v\in \tn} X_{n,v} \, E_v} |\tn\right] 
=\expp{ \sum_{v\in \tn} (\lambda X_{n,v} -
   \log(1+\lambda X_{n,v}))}.
\]
For $x\geq 0$, we have $\frac{x^2}{2} - \frac{x^3}{3} \leq  x - \log(1+x) \leq
\frac{x^2}{2}$. Thanks to Theorem \ref{theo:principal}, we have:
\[
\sum_{v\in \tn} X_{n,v}^2=A_n(x f^2) 
\,\xrightarrow[n\rightarrow+\infty]{\text{a.s.}}\,
\sqrt{2\alpha}\,  \Phi_e(xf^2) 
\]
and
\[
\sum_{v\in \tn} X_{n,v}^3=|\tn|^{-1/4} A_n(x^2 f^3) 
\,\xrightarrow[n\rightarrow+\infty]{\text{a.s.}}\,
0.
\]
We                   deduce                  that                   a.s.
$\lim_{n\rightarrow+\infty } \E\left[\expp{-\lambda Z_n(f)}|\tn\right] =
\exp{(      \lambda^2      \sqrt{2\alpha}\,     \Phi_e(xf^2)      /2)}$.
Let         $K>0$,         and        consider         the         event
$  B_K=\bigcap  _{n\in  \N}\{A_n(xf^2)\leq K\}$.  Since on $B_{K}$,
the  term $\E\left[\expp{-\lambda  Z_n(f)}|\tn\right]  $  is bounded  by
$\exp(\lambda^2 K/2)$, we  deduce from dominated convergence
that  for any  continuous  bounded function  $g$ on  the  set of  finite
measure on $[0,1]$ (endowed with  the topology of the weak convergence),
we have:
\begin{align*}
\lim_{n\rightarrow+\infty } 
\E\left[\expp{-\lambda Z_n(f)} g(A_n)\ind_{B_{K}}\right]
&=\lim_{n\rightarrow+\infty } 
\E\left[\E\left[\expp{-\lambda Z_n(f)}|\tn \right]
  g(A_n)\ind_{B_{K}}\right] \\
&= \E\left[\expp{   \lambda^2 \sqrt{2\alpha}\,
  \Phi_e(xf^2) /2}  g(\sqrt{2\alpha}\, \Phi_e)
\ind_{B_{K}}\right] \\
&= \E\left[\expp{  -\lambda (2 \alpha )^{1/4}\sqrt{ \Phi_e(xf^2)}
  \,\, G}  g(\sqrt{2\alpha}\,\Phi_e) 
\ind_{B_{K}}\right] ,
\end{align*}
where $G$ is a standard Gaussian random variable independent of $e$. 
We deduce that the convergence  in distribution \reff{eq:cv-ZA} holds
conditionally on $B_K$. 
Since  $A_n(xf^2)$  is
finite for every  $n$ and converges a.s. to  a finite limit, we  get that for
any  $\varepsilon>0$,  there  exists $K_\varepsilon$  finite  such  that
$\P(B_{K_\varepsilon})\geq 1-\varepsilon$. Then use Lemma
\ref{lem:cv-LP} below to conclude that \reff{eq:cv-ZA} holds for $f$ 
non-negative. 

In the general  case, we set $f_+=\max(0, f) $  and $f_-=\max(0, -f)$ so
that  $f=f_+-f_-$. Notice  that  $f_+$ and  $f_-$  are non-negative  and
continuous. We have proved that  \reff{eq:cv-ZA} holds with $f$ replaced
by  $\lambda_+  f_+ +  \lambda_-  f_-$  for  any $\lambda_+\geq  0$  and
$\lambda_-\geq  0$.  Since  $f_+  f_-=0$,  this  implies  the  following
convergence in distribution:
\[
(Z_n(f_+), Z_n(f_-), A_n)
\,\xrightarrow[n\rightarrow+\infty]{(d) } \,
\Big( (2 \alpha )^{1/4} \sqrt{ \Phi_e(xf_+^2)} \,\, G_+,  (2 \alpha
)^{1/4}\sqrt{ \Phi_e(xf_-^2)} \,\, G_-, \sqrt{2 \alpha}\, \, 
\Phi_e\Big),
\]
where $G_+$ and $G_-$ are independent standard  Gaussian random variables
independent of $e$.  Then, using again that $f_+ f_-=0$, we obtain that,
conditionally on $e$, $ \sqrt{ \Phi_e(xf_+^2)} \,\, G_+ - \sqrt{
  \Phi_e(xf_-^2)} \,\, G_-$ is distributed as $ \sqrt{ \Phi_e(xf^2)}
\,\, G$, where $G$ is a standard Gaussian random variable
independent of $e$. We deduce that  \reff{eq:cv-ZA} holds. This ends the proof.
\end{proof}

\begin{lem}
   \label{lem:cv-LP}
   Let  $(\Gamma_\varepsilon, \varepsilon>0)$  be a  sequence of  events
   such                                                             that
   $\lim_{\varepsilon  \rightarrow  0}  \P(\Gamma_\varepsilon)=1$.   Let
   $(W_n,  n\in \N)$  and $W$  be random  variables taking  values in  a
   Polish   space   $\cm$.   Assume  that   for   all   $\varepsilon>0$,
   conditionally on $\Gamma_\varepsilon$, the  sequence $(W_n, n\in \N)$
   converges  in  distribution  towards   $W$.  Then  $(W_n,  n\in  \N)$
   converges in distribution towards $W$.
\end{lem}
\begin{proof}
  Let  $g$  be a  real-valued  bounded  continuous function  defined  on
  $\cm$.       It       is        enough       to       prove       that
  $\lim_{n\rightarrow+\infty   }   |\E[g(W_n)]    -   \E[g(W)]|=0$.   By
  hypothesis, we have that for all $\varepsilon>0$:
\[
\lim_{n\rightarrow+\infty } \E[g(W_n)|\Gamma_\varepsilon]
=\E[g(W)|\Gamma_\varepsilon].
\]
We get:
\[
   |\E[g(W_n)] - \E[g(W)]|
\leq   |\E[g(W_n)|\Gamma_\varepsilon] - \E[g(W)|\Gamma_\varepsilon]|
  \P(\Gamma_\varepsilon)
 + 2 \norm{g}_\infty \P(\Gamma_\varepsilon^c)
\]
We deduce that $\limsup_{n\rightarrow+\infty }  |\E[g(W_n)] - \E[g(W)]|
\leq  2 \norm{g}_\infty \P(\Gamma_\varepsilon^c)$. Since
$\lim_{\varepsilon    \rightarrow   0}    \P(\Gamma_\varepsilon^c)=0$,
we deduce that $\lim_{n\rightarrow+\infty }  |\E[g(W_n)] -
\E[g(W)]|=0$. This ends the proof. 
\end{proof}

\subsection{Proof of Proposition \ref{prop:fluctuation}}
We deduce   Proposition
\ref{prop:fluctuation} directly from  Lemmas \ref{lem:cvD} and \ref{lem:cvD2}
below. 

Using notations from Section
\ref{sec:premlim-lem}, we set:
\[
\Delta_{n}=\inv{2\sqrt{\alpha} } |\tn|^{1/4}
(A_{1,n}-A_{2,n}).
\]

\begin{lem}
   \label{lem:cvD}
   Let $f\in \cc([0,1])$ be locally Lipschitz continuous on $(0,1]$ such
   that $\norm{x^a f'}_{\text{esssup}} $ is finite for some $a\in (0,1)$. We have
   the following convergence in probability:
\[
|\tn|^{1/4} (A_n- \sqrt{2 \alpha}\, \Phi_e)(f) -2\sqrt{\alpha} \,  \Delta_n(f)
\,\xrightarrow[n\rightarrow+\infty]{\P}\, 
0.
\]
\end{lem}
\begin{proof}
We keep notations from Section
\ref{sec:premlim-lem}. We have:
\[
\val{|\tn|^{1/4} (A_n- \sqrt{2 \alpha}\, \Phi_e)(f) - 2\sqrt{\alpha} \,
  \Delta_n(f)}\leq  
\Delta_{1,n}+\Delta_{3,n}+ \Delta_{4,n} + \Delta_{5,n},
\]
where 
\begin{align*}
&\Delta_{1,n}=|\tn|^{1/4} \lvert A_{n}(f)-A_{1,n}(f)\rvert, 
\quad
\Delta_{3,n}=|\tn|^{1/4} \lvert A_{2,n}(f)-A_{3,n}(f)\rvert, \\
&\Delta_{4,n}=|\tn|^{1/4} \lvert A_{3,n}(f)-A_{4,n}(f)\rvert, 
\quad
\Delta_{5,n}=|\tn|^{1/4} \lvert A_{4,n}(f)-\sqrt{2 \alpha}\,
\Phi_e(f)\rvert.   
\end{align*}
Using  Lemmas \ref{lem:A1}, \ref{lem:A3} and \ref{lem:A4} part (i), we deduce the following
convergence in probability:
\[
\Delta_{1,n}\,\xrightarrow[n\rightarrow+\infty]{\P} \,0,
\quad
\Delta_{3,n}\,\xrightarrow[n\rightarrow+\infty]{\P} \,0
\quad\text{and}\quad
\Delta_{4,n}\,\xrightarrow[n\rightarrow+\infty]{\P} \,0.
\]

We study the convergence of $\Delta_{5,n}$. We set:
\[
I_n=\inv{n+1}\sum_{k=1}^{n+1}
\int_0^{e(U_k)}  dr\,  f(\sigma_{r,U_k})- 
\int_0^1 ds \int_0^{e(s)}   dr\,  f(\sigma_{r,s}) .
\]
By conditioning with
respect to $e$, we deduce that: 
\begin{equation}
   \label{eq:In2}
\E[I_n^2]\leq  \inv{n+1} \E\left[\left(\int_0^{e(U_1)}  dr\,
  f(\sigma_{r,U_1})\right)^2 \right]
\leq  \frac{\norm{f}_\infty^2}{n+1}  \, \E\Big[\int_0^1 ds\,  e(s)^2\Big].      
\end{equation}
Using the definition of $A_{4,n}(f)$, we get $\Delta_{5,n}\leq 
\Delta_{6,n}+ \sqrt{2 \alpha} \, \Delta_{7,n}$ 
with 
\[
\Delta_{6,n}= |\tn|^{1/4} \left|1- \frac{|\tn|^{3/2} } {2(n+1)\sqrt{2n}}\right|
A_{4,n}(|f|)
\quad\text{and} \quad
\Delta_{7,n}= |\tn|^{1/4} |I_n| .
\] 
From the a.s. convergence of $A_{4,n}(|f|)$ towards a finite limit, see
Lemma \ref{lem:LGN}, we deduce that a.s. $\lim_{n\rightarrow+\infty } 
\Delta_{6,n}=0$. 
Since $\E\Big[\int_0^1 ds\,  e(s)^2\Big]$ is finite, see \cite{r:qeem},
we deduce from \reff{eq:In2}
that $\lim_{n\rightarrow+\infty } \E[\Delta_{7,n}^2]=0$. We obtain  that:
\[
\Delta_{5,n}\,\xrightarrow[n\rightarrow+\infty]{\P} \,0.
\]

Then, we collect all the convergences together to get the result.    
\end{proof}

Now, we study the convergence in distribution of $\Delta_n(f)$. 
\begin{lem}
   \label{lem:cvD2}
Let $f\in \cc([0,1])$ be locally Lipschitz  continuous on $(0,1]$ such
that $\norm{x^a f'}_{\text{esssup}} $ is finite for some $a\in (0,1)$. We have
the following convergence in distribution: 
\begin{equation}
   \label{eq:cv-AD}
(2\sqrt{\alpha}\, \Delta_n(f), A_n)
\,\xrightarrow[n\rightarrow+\infty]{\text{(d)}}\, 
\left((2 \alpha)^{1/4 }\sqrt{ \Phi_e(xf^2)} \,\, G, \sqrt{2 \alpha}\, \, \Phi_e\right),
\end{equation}
where $G$ is a standard Gaussian random variable independent of $e$. 
\end{lem}

\begin{proof}
According
to Lemma \ref{lem:hh}, we get that $(\Delta_n(f), A_n)$ is distributed as
$( \Delta_n'(f), A_n)$ where:
\[
\Delta'_n(f)= |\tn|^{-5/4}\sum_{v\in
  \tn }|\tnv|f\left(\frac{| \tnv |}{|\tn|}\right) \, Y'_{n,v},
\quad\text{with}\quad
Y'_{n,v}=\sqrt{n} \left(\E\left[\frac{L'_n E_v}{S_{\tn}}\right]
-\frac{L'_n E_v}{S_{\tn}} \right),
\]
and $S_{\bt}=\sum_{v\in  \bt}E_v$ for $\bt\in  \T$, with $L'_n$  a random
variable distributed as $L_n$, and thus with  density  given by
\reff{eq:dens-Ln},  independent
of $\tn$ and $(E_u, u\in  \cu)$ independent exponential random variables
with  mean 1,  independent of  $L'_n$ and  $\tn$.  So it is enough to
prove \reff{eq:cv-AD} with $\Delta_n$ replaced by $\Delta'_n$.

Recall  the definition
\reff{eq:def-Zn} of $Z_n(f)$. Since  $L'_n$ is  independent of $(E_u,
u\in \cu)$ and $\tn$, we get:
\[
\Delta'_n(f)=
\frac{\sqrt{n}}{\sqrt{|\tn|}} (\kappa_{1,n}+\kappa_{2,n})  A_n(f)
-\sqrt{n} \, 
\frac{L'_n}{S_{\tn}} Z_n (f) 
\]
with
\[
\kappa_{1,n}=|\tn|^{3/4}\,  (\E[L'_n] -L'_n)
\E\left[\frac{E_\emptyset}{S_{\tn}}\right]
\quad\text{and}\quad
\kappa_{2,n}= |\tn|^{3/4}\,  L'_n\left(\E\left[\frac{ E_\emptyset}{S_{\tn}}\right]
-\inv{S_{\tn}} \right) .
\]
Thanks to Corollary \ref{cor:moments_exp_gam} with $\alpha=\gamma=1$ and $\beta=0$, we
have that:
\[
\E[E_\emptyset/S_{\tn}]={\Gamma(2n+1)}/{\Gamma(2n+2)}=1/|\tn|.
\] 
Using \reff{eq:moments_Ln2}, we get:
\[
\E[|\kappa_{1,n}|]\leq  |\tn|^{3/4} \sqrt{\Var(L'_n)}\,\, 
\frac{\Gamma(2n+1)}{\Gamma(2n+2)}
\leq  \inv{\sqrt{\alpha}}\,\, \inv{(2n+1)^{1/4}}\cdot
\]
 We deduce that $\lim_{n\rightarrow+\infty } \kappa_{1,n }=0$ in
 probability. 
Using  \reff{eq:moments_Ln1} and Corollary \ref{cor:moments_exp_gam} (three times), we get:
\begin{align*}
\E[\kappa_{2,n}^2]
&= |\tn|^{3/2} \frac{n+1}{\alpha} \left( 
\frac{\Gamma(2n+1)^2}{\Gamma(2n+2)^2} +
\frac{\Gamma(2n-1)}{\Gamma(2n+1)} - 2 \frac{\Gamma(2n+1)}{\Gamma(2n+2)}
\frac{\Gamma(2n)}{\Gamma(2n+1)}
\right)\\
&=\frac{ |\tn|^{3/2} \,\,(n+1)(2n+3)}{\alpha 2n (2n+1)^2 (2n-1)}\cdot
\end{align*}
 We deduce that $\lim_{n\rightarrow+\infty } \kappa_{2,n }=0$ in
 probability. 

We deduce from the law of large numbers that  $\lim_{n\rightarrow+\infty
} S_{\tn}/|\tn|=1$ in probability. 
According to \cite{ad:13}, we have that a.s. $\lim_{n\rightarrow+\infty
} L_n/\sqrt{n}= 1/\sqrt{\alpha}$. This implies the following convergence in
probability  $\lim_{n\rightarrow+\infty
} L'_n/\sqrt{n}= 1/\sqrt{\alpha}$. We obtain  that:
\[
\sqrt{n}\, 
\frac{L'_n}{S_{\tn}}
\,\xrightarrow[n\rightarrow+\infty]{\P } \,  \inv{2\sqrt{\alpha}} \cdot
\]

We deduce  that $( 2 \sqrt{\alpha} \,  \Delta'_n(f), A_n)$ has  the same
limit in distribution as $(- Z_n(f), A_n)$ as $n$ goes to infinity. Then
use Lemma \ref{lem:cv-ZA} 
to get that  \reff{eq:cv-AD} holds with $\Delta_n$ replaced by $\Delta'_n$. 
This ends the proof of the Lemma.
\end{proof}

\section{Proof of Corollary \ref{cor:cv-measure}}
\label{sec:proof-cv-measure}
Before stating the proof, we recall the definition of the contour
process of a discrete rooted ordered tree, see \cite{dlg:rtlpsbp}. 

\subsection{Contour process}
\label{sec:contour}
Let    $\bt\in     \T$    be    a finite   tree.     The    contour    process
$C^\bt=(C^\bt(s), s\in  [0, 2 |\bt|])$ is  defined as the distance  to the
root of a particle visiting continuously each edge of $\bt$ at speed one
(where all edges  are of length 1) according to  the lexicographic order
of the nodes. More precisely, we set $\emptyset=u(0)<u(1)< \ldots <
u(|\bt|-1)$  the nodes of $\bt$ ranked in the lexicographic
order. By convention, we set $u(|\bt|)=\emptyset$. 

We      set       $\ell_0=0$,      $\ell_{|\bt|+1}=2$       and      for
$k\in  \{1,   \ldots,  |\bt|\}$,   $\ell_k=d(u(k-1),  u(k))$.    We  set
$L_k=\sum_{i=0}^k  \ell_i$  for  $k\in   \{0,  \ldots,  |\bt|+1\}$,  and
$L'_k=L_k+d(u(k),           \mrca(u(k),          u(k+1)))$           for
$k\in \{0,  \ldots, |\bt|-1\}$. (Notice  that $L'_k=L_k$ if and  only if
$u(k)\prec    u(k+1)$.)     We    have   $L_{|\bt|}=2|\bt|-2    $    and
$L_{|\bt|+1}=2|\bt| $. We define for $k\in \{0, \ldots, |\bt|-1\}$:
\begin{itemize}
   \item  for $s\in
[L_k, L'_{k})$, the particle goes down from $u(k)$ to $\mrca(u(k),
u(k+1))$:  $C^\bt (s)=|u(k)|-(s-L_k)$;
\item for $s\in
[L'_{k}, L_{k+1})$, the particle goes up from  $\mrca(u(k),
u(k+1))$ to $u(k+1)$:
$C^\bt (s)=|\mrca(u(k),  u(k+1))|+(s-L'_k)$,
\end{itemize}
and $C^\bt(s)=0$ for  $s\in [2|\bt|-2, 2|\bt|]$. 
Notice that $C^\bt$ is continuous. 

For $u\in \bt$, we define 
$\ci_u$ the time interval during which the
particle explores the edge attached below $u$. More precisely for $k\in
\{1, \ldots, |\bt|-1\}$, we set:
\[
\ci_{u(k)}=[L_k-1, L_{k}) \bigcup [L''_k, L''_k+1),
\]
where  $L''_k=\inf\{s\geq  L_{k}, \, C^\bt (s)<|u(k)|\}$ and
$\ci_\emptyset=[2|\bt|-2, 2|\bt|]$. The sets $(\ci_u, u\in \bt)$ are
disjoints 2 by 2 with $\bigcup _{u\in \bt} \ci_u=[0, 2|\bt|]$. For $u\in \bt$, we have that the
Lebesgue measure of $\ci_u$ is 2 and
\begin{equation}
   \label{eq:IuCd}
C^\bt(s)\leq  d(\emptyset, u)\leq  C^\bt(s) +1 \quad\text{for all $s\in \ci_u$}.
 \end{equation}

\subsection{Elementary functionals of finite trees}
\label{sec:approx}
Let $\bt\in \T$ be a finite tree and $k\in \N^*$. For $\bu=(u_1, \ldots, u_k)\in \bt^k$, we
define $\mrca(\bu)=\mrca(\{u_1, \ldots, u_k\})$ the most recent common
ancestor of $u_1, \ldots, u_k$. 
We consider the following elementary functional of a tree, defined for $\bt\in \T$:
\begin{equation}
   \label{eq:def-D}
D_k(\bt)=\sum_{\bu\in \bt^k} d(\emptyset, \mrca(\bu)). 
\end{equation}
We have: 
\begin{equation}
   \label{eq:bt=D}
\sum_{v\in \bt} |\bt_v|^k =D_k(\bt)+ |\bt|^k,
\end{equation}
which  we obtain  from the following equalities
\[
\sum_{v\in \bt} |\bt_v|^k
=\sum_{v\in \bt}\,\, \sum_{\bu \in  \bt^k} \ind_{\{v\preccurlyeq  \mrca(\bu)\}}
= \sum_{\bu\in  \bt^k}\,\, \sum_{v\in \bt}\ind_{\{v\preccurlyeq  \mrca(\bu)\}}
=\sum_{\bu\in  \bt^k} (d(\emptyset, \mrca(\bu))+1). 
\]

For      $x=(x_1,     \ldots,      x_k)\in     \R^k$,      denote     by
$(x_{(1)},  \ldots, x_{(k)})$  its  order statistic  which is  uniquely
defined     by     $x_{(1)}\leq     \cdots     \leq     x_{(n)}$     and
$\sum_{i=1}^k  \delta   _{x_i}=\sum_{i=1}^k  \delta   _{x_{(i)}}$,  with
$\delta_z$  the  Dirac mass  at  $z$.  Recall notation  $m_h(s,t)$,  see
\reff{eq:def-m}, for  the minimum of  $h$ over the interval  with bounds
$s$ and $t$.  We set:
\begin{equation}
   \label{eq:def-DD}
\cd_k(\bt)=\int_{[0, |\bt|]^k}   m_{C^\bt}(2x_{(1)}, 2x_{(k)})\, dx ,
\end{equation}
with      the       conventions      that      if       $k=1$,      then
$\cd_1(\bt)=\int_{[0,  |\bt|]}   C^\bt(2x)\, dx$.

We have the following lemma. 

\begin{lem}
   \label{lem:d=D}
We have for $\bt\in \T$ and $k\in \N^*$: 
\begin{equation}
   \label{eq:DrD}
0 \leq  D_k(\bt) - \cd_k(\bt) \leq  |\bt|^k.
\end{equation}
\end{lem}

\begin{proof}
For $\bu=(u_1, \ldots, u_k)\in \bt^k$, we have the following
generalization of  \reff{eq:IuCd}: 
 for $x=(x_1, \ldots, x_k)\in 
\prod_{i=1}^k \ci_{u_i}$, 
\[
m_{C^\bt}(x_{(1)}, x_{(k)}) \leq  d(\emptyset, \mrca(\bu)) \leq
m_{C^\bt}(x_{(1)}, x_{(k)}) +1. 
\]
(Notice that $m_{C^\bt}(x_{(1)}, x_{(k)}) =  d(\emptyset,
\mrca(\bu)) $ as soon as  $\mrca(\bu)\prec u_i$
for all $i\in \{1, \ldots, k\}$.) We deduce that:
\[
0\leq 2^k d(\emptyset,\mrca(\bu))-  \int _{\prod_{i=1}^k \ci_{u_i}} 
m_{C^\bt}(x_{(1)}, x_{(k)})  \, dx  \leq   2^k.
\]
By summing over $\bu\in \bt^k$, we get:
\[
0\leq  2^k D_k(\bt)- \int_{[0, 2|\bt|]^k}   m_{C^\bt}(x_{(1)},
x_{(k)})\, dx  \leq  2^k |\bt|^k
.
\]
Use the change of variable $2y=x$ to get \reff{eq:DrD}. 
\end{proof}

\subsection{Convergence of contour processes}
\label{sec:td}
We assume  that $\rp$ is  a probability  distribution on $\N$  such that
$1>\rp(1)+\rp(0)\geq \rp(0)>0$ and which is critical  (that is $\sum_{k\in \N} k \rp(k)=1$).
 We also assume
that  $\rp$  is in  the  domain  of  attraction  of a  symmetric  stable
distribution of Laplace exponent $\psi(\lambda)=\kappa \lambda ^\gamma$
with  $\gamma\in  (1,2]$ and $\kappa>0$,  and  renormalizing  sequence
$(a_p,  p\in  \N^*)$   of  positive  reals:  if   $(U_k,k\in  \N^*)$  are
independent  random variables  with the  same distribution  $ \rp$,  and
$W_p=\sum_{k=1}^p U_k  - p$, then $W_p/a_p$  converges in distributions,
as $p$  goes to  infinity, towards  a random  variable $X$  with Laplace
exponent $-\psi$  (that is  $\E[\expp{-\lambda X}]=\expp{\psi(\lambda)}$
for $\lambda\geq 0$). Notice this convergence implies that:
\begin{equation}
   \label{eq:cv-an}
\lim_{p\rightarrow+\infty } \frac{a_p}{p}=0.
\end{equation}
\begin{rem}
   \label{rem:s2-fini-0}
If $\rp$ has  finite variance, say $\sigma^2$, then one can take
$a_p=\sqrt{p}$ and $X$ is then a centered Gaussian random variable
with variance $\sigma^2$, so that $\psi(\lambda)= \sigma^2\lambda  ^2/2$. 
\end{rem}

The  main  theorem  in  Duquesne  \cite{d:ltcpcgwt}  on  the  functional
convergence in distribution of the  contour process stated when $\rp$ is
aperiodic, can easily  be extended to the case  $\rp$ periodic.  (Indeed
the  lack of  periodicity  hypothesis is  mainly used  in  Lemma 4.5  in
\cite{d:ltcpcgwt} which is based on Gnedenko local limit theorem.  Since
the latter  holds \textit{a fortiori}  for lattice distributions  in the
domain of  attraction of stable law,  it allows to extend  the result to
such periodic distribution,  as soon as one uses  sub-sequences on which
the conditional probabilities are well defined.) It will be stated in this
more general version, see Theorem \ref{theo:d} below.  Since the contour
process  is  continuous  as  well  as  its  limit,  the  convergence  in
distribution  holds  on  the   space  $\cc([0,1])$  of  real  continuous
functions endowed with the supremum norm.

\begin{theo}
   \label{theo:d}
   Let $\rp$ be  a critical probability  distribution on $\N$,
   with  $1>\rp(1)+\rp(0)\geq \rp(0)>0$, which  belongs to  the  domain of  attraction of  a
   symmetric     stable     distribution     of     Laplace     exponent
   $\psi(\lambda)=  \kappa \lambda  ^\gamma$   with   $\gamma\in  (1,2]$
   and $\kappa>0$,    and 
   renormalizing sequence $(a_p,  p\in \N^*)$.  Let $\tau$ be  a GW tree
   with offspring distribution $\rp$, and $\tau^{(p)}$ be distributed as
   $\tau$  conditionally  on  $\{|\tau|=p\}$.   There  exists  a  random
   non-negative continuous process $H=(H_s,  s\in [0,1])$, such that the
   following convergence on the space $\cc([0,1])$  holds in distribution:
\[
\frac{a_p}{p}\left(C^{\tau^{(p)}}(2ps), s\in [0,1]\right) 
\,\xrightarrow[p\rightarrow+\infty]{(d) } \,
H,
\]
where the convergence is taken  along the
infinite  sub-sequence of  $p$  such that  $\P(|\tau|=p)>0$. 
\end{theo}

The process $H$, see \cite{d:ltcpcgwt} for a construction of $H$, is 
the so called  normalized excursion for  the height
process,  introduced in \cite{lglj:bplp}, 
of a Lévy tree with branching mechanism $\psi$.
\begin{rem}
   \label{rem:s2-fini-1}
If $\psi(\lambda)= \alpha\lambda  ^2$, for some $\alpha>0$, then $H$ is
distributed as $\sqrt{2/\alpha}\, B$, where $B$ is the positive Brownian
excursion, see \cite{dlg:rtlpsbp}. 
\end{rem}

\subsection{Convergence of additive functionals}
We now give the main result of this Section. 
\begin{cor}
   \label{cor:cvDk}
   Under the hypothesis and notations of Theorem \ref{theo:d}, we  have the following
   convergences in distribution 
for all $k\in \N^*$:
\[
\lim_{p\rightarrow+\infty }
\frac{a_p}{ p^{k+1}} \sum_{v\in \tau^{(p)}} |\tau^{(p)}_v|^{k} 
\,\stackrel{(d)}{=}\,
\lim_{p\rightarrow+\infty }
\frac{a_p}{ p^{k+1}}\sum_{\bu\in (\tau^{(p)})^k} d(\emptyset, \mrca(\bu))
\,\stackrel{(d)}{=}\,
\int_0^1 ds \int_0^{H(s)} dr \, \sigma_{r,s}(H)^{k-1}, 
\]
where $\sigma_{r,s}(H)$  is the  length of the  excursion of  the height
process $H$  above $r$ straddling  $s$ defined in  \reff{eq:def-srs} and
where the  convergence is taken  along the infinite sub-sequence  of $p$
such that $\P(|\tau|=p)>0$.
\end{cor}
We deduce from their proofs,  using the Skorohod representation theorem, that
all the  convergences in  distribution of Corollary  \ref{cor:cvDk} hold
simultaneously for all $k\in \N^*$.

\begin{proof}
  Recall   notation   $m_h(s,t)$    and   $\sigma_{r,s}(h)$   given   in
  \reff{eq:def-m} and \reff{eq:def-srs}.  We shall take limits along the
  infinite sub-sequence of $p$ such that $\P(|\tau|=p)>0$.

Recall definitions  \reff{eq:def-D} of  $D_k(\bt)$ and \reff{eq:def-DD}
of  $\cd_k(\bt)$. Thanks to Lemma \ref{lem:d=D} and  \reff{eq:cv-an} which
implies  that
$(p^{-(k+1)} a_p (D_k(\tau^{(p)}) - \cd_k(\tau^{(p)})),\, p\in \N^*)$
converges in 
probability towards 0 and to \reff{eq:bt=D}, we see the proof of the
corollary is complete as
soon as we obtain that for all $k\in \N^*$:
\begin{equation}
   \label{eq:cor-res}
\lim_{p\rightarrow+\infty }
\frac{a_p}{ p^{k+1}} \cd_k(\tau^{(p)})
\,\stackrel{(d)}{=}\,
\int_0^1 ds \int_0^{H(s)} dr \, \sigma_{r,s}(H)^{k-1}. 
\end{equation}

We deduce from Theorem \ref{theo:d} the following convergence in law:
\[
\frac{a_p}{ p^{2}} \cd_1(\tau^{(p)}) 
=\frac{a_p}{ p^{2}} \int_{[0,  p]}    C^{\tau^{(p)}}(2x)\, dx
=\int_{[0,  1]}  \frac{a_p}{ p}    C^{\tau^{(p)}}(2ps)\, ds
\xrightarrow[p\rightarrow+\infty]{(d) }  \int_{[0,1]}
  ds\,   H(s) .
\]
This gives \reff{eq:cor-res} for $k=1$. 

Thanks to equality \reff{eq:=} with $a=k-1$, we have that for $k\geq 2$ and $\bt\in \T$: 
\begin{align*}
2\cd_k(\bt)
&=k(k-1)\int_{[0, |\bt|]^2}  dx_1dx_2 \,
 |x_2-x_1|^{k-2}\,m_{C^\bt(2\bullet)} (x_1, x_2)\\
&=k(k-1)|\bt|^k\int_{[0, 1]^2}  dx_1dx_2 \,
  |x_2-x_1|^{k-2}\,m_{C^\bt(2|\bt|\bullet)} (x_1, x_2).
\end{align*}
We deduce from Theorem \ref{theo:d} the following convergence in law for
all $k\in \N^*$ such that $k\geq 2$:
\[
\frac{a_p}{ p^{k+1}} \cd_k(\tau^{(p)}) 
\xrightarrow[p\rightarrow+\infty]{(d) } \frac{k(k-1)}{2} \int_{[0,1]^2}
  dsds' \,
 |s'-s|^{k-2}\,m_H(s,s').
\]
Then use  \reff{eq:I=J} from Lemma \ref{lem:sigma-moment} to get \reff{eq:cor-res}. This ends the proof. 
\end{proof}

\section{Appendix}
\label{sec:appendix}

\subsection{Upper bounds for moments of the cost functional}
According to  \cite{fk:ld}, for  $\beta>\frac{1}{2}$ and  $k\in\N^{*}$,  
there  exists  a  finite  constant   $C_{k,\beta}$  such  that  for  all
$n\in\N^*$,
\begin{equation}\label{moment_ordre_p}
 \E\left[\left(\sum_{v\in \tn}\lvert \tnv\rvert^{\beta}\right)^{k}
 \right]\le C_{k,\beta}\,|\tn|^{k(\beta+\frac{1}{2})}. 
\end{equation}
(Notice that \reff{moment_ordre_p} is stated in \cite{fk:ld} with
$\tnv^*=\tnv\backslash \cl(\tnv)$ instead of $\tnv$; but using that
$|\tnv|= 2|\tnv^*|+1 $ it is elementary to get \reff{moment_ordre_p}.)

The  following lemma,  which plays  a key  role in  the proofs  of Lemmas
\ref{lem:A1} and  \ref{lem:A2}, is  a direct consequence  of these
upper bounds.
\begin{lem}\label{lem_mom_1}
For all $a\in [0, 1/2)$ and  $f\in\cb([0,1])$, we have for $k\in\N^{*}$:
 \begin{align}
   \E\left[|A_{n}(f)|^{k}\right]
&\le C_{k,1-a}\lVert x^{a}f\rVert_{\infty}^{k}, \label{mom_1}\\
   \E\left[A_{n}(xf^{2})\right]
&\le C_{1,2-2a}\lVert x^{a}f\rVert_{\infty}^{2}. \label{mom_3}
 \end{align}
\end{lem}

\begin{proof}
Let $k\in\N^*$. Using (\ref{moment_ordre_p}), we have:
\[
 \E\left[|A_{n}(f)|^{k}\right]
\le |\tn|^{-\frac{3}{2}k}\lVert x^{a}f\rVert_{\infty}^{k}\,\E\left[\left(\sum_{v\in \tn}\frac{|\tnv|^{1-a}}{|\tn|^{-a}}\right)^{k}\right]
\le C_{k,1-a}\lVert x^{a}f\rVert_{\infty}^{k},
\]
which gives \reff{mom_1}. Moreover, we also have:
\[
\E\left[A_{n}(xf^{2})\right]
\le |\tn|^{-\frac{3}{2}}\lVert x^{a}f\rVert_{\infty}^{2}\,\E\left[\sum_{v\in \tn}\frac{|\tnv|^{2-2a}}{|\tn|^{1-2a}}\right]
\le C_{1,2-2a}\lVert x^{a}f\rVert_{\infty}^{2}\,
\]
and we get \reff{mom_3}.
\end{proof}

\subsection{A lemma for binomial random variables}
We give a lemma used  for the proof of Lemma \ref{lem:A4}.
 
\begin{lem}\label{lem_binomial}
  Let $X$ be a binomial random variable with parameter $(n,p)\in
  \N^*\times  (0,1)$.
  \begin{enumerate}
   \item[(i)] For $a\in (0,1]$, we have \[\E\left[\left(2X+1\right)^{-a}\right]\le \left(1 \wedge \frac{1}{p(n+1)}\right)^{a}\cdot\]
   \item[(ii)] Let  $f\in \mathcal{C}((0,1])$ be locally Lipschitz
     continuous and $b\in (0,1)$.
  Then we have:
\[
\E\left[\left|f\left(\frac{2X+1}{2n+1}\right)-f(p)\right|\right]
\le \frac{\norm{x^{b}f'} _{\text{esssup}}}{1-b}
\left(p^{-\frac{b}{2}}+p^{\frac{1}{2}-b}\right) \, n^{-1/2}.
\]
  \end{enumerate}
 \end{lem}
                                            
 \begin{proof}
We prove (i). Let $a\in (0,1]$.
  Let $X$ be a binomial random variable with parameter $(n,p)$.  An
  elementary computation gives that:
 \begin{equation}\label{moment_neg}
  \E\left[\frac{1}{1+X}\right]=\frac{1-(1-p)^{n+1}}{p(n+1)}\cdot
 \end{equation}
  Using Jensen inequality and \reff{moment_neg}, we get
\[
\E\left[\left(\frac{1}{2X+1}\right)^{a}\right]
\le \E\left[\frac{1}{2X+1}\right]^{a}
\le \E\left[\frac{1}{1+X}\right]^{a}
   \le \left(1 \wedge \frac{1}{p(n+1)}\right)^{a}.
\]   

We prove (ii).  Let $b\in(0,1)$. We have $ \left|f\left(\frac{2X+1}{2n+1}\right)-f(p)\right|\le \norm{
    x^{b}f'}_{\text{esssup}}
\left|\int_{p}^{\frac{2X+1}{2n+1}}x^{-b}dx\right|  $ and thus
\begin{equation}\label{ineq_1}
  \left|f\left(\frac{2X+1}{2n+1}\right)-f(p)\right|\le
  \frac{\norm{x^{b}f'}_{\text{esssup}}}{1-b}
  \left|\left(\frac{2X+1}{2n+1}\right)^{1-b}\!\!\!\!-p^{1-b}\right|  .  
 \end{equation}
  We decompose the right-hand side term into two parts:
  \begin{equation}\label{ineq_2}
   \left|\left(\frac{2X+1}{2n+1}\right)^{1-b}-p^{1-b}\right|\le \left|p^{1-b}-\left(\frac{X}{n}\right)^{1-b}\right|
                                                                 +\left|\left(\frac{2X+1}{2n+1}\right)^{1-b}-\left(\frac{X}{n}\right)^{1-b}\right|\cdot 
  \end{equation}
We shall use the following key inequality: for all $x,y>0$  and $0<b<1$, we have:
\begin{equation}\label{ineq_3}
 |x^{1-b}-y^{1-b}|\le x^{-b}|x-y|.
\end{equation}
For  the first  term  of the  right hand  side  of \reff{ineq_2},  using
\reff{ineq_3},                          we                          have
$\left|p^{1-b}-\left(\frac{X}{n}\right)^{1-b}\right|\le
p^{-b}\left|p-\frac{X}{n}\right|$. Hence, we get:
\begin{equation}
\label{ineq_4}
 \E\left[\left|p^{1-b}-\left(\frac{X}{n}\right)^{1-b}\right|\right]
\leq
 p^{-b} \sqrt{\Var\left(X/n\right)}
\leq  p^{\frac{1}{2}-b}n^{-1/2}.
\end{equation}
For  the second   term  of the  right hand  side  of \reff{ineq_2},
using \reff{ineq_3} again, we get: 
\[
\left|\left(\frac{2X+1}{2n+1}\right)^{1-b}-\left(\frac{X}{n}\right)^{1-b}\right|
\leq  \left(\frac{2X+1}{2n+1}\right)^{-b}
\left|\frac{2X+1}{2n+1}-\frac{X}{n}\right|
\leq  \frac{(2n+1)^{b-1}}{(2X+1)^{b}}\cdot
\]
This gives, using (i) and $|1 \wedge (1/x)|^b\leq  x^{-b/2}$ for $x>0$,
that: 
\begin{equation}
\label{ineq_5}
\E\left[\left|\left(\frac{2X+1}{2n+1}\right)^{1-b}-\left(\frac{X}{n}\right)^{1-b}\right|\right] 
  \le (2n+1)^{b-1}p^{-\frac{b}{2}}(n+1)^{-\frac{b}{2}}
  \le p^{-\frac{b}{2}}n^{-1/2}.
\end{equation}
 Using \reff{ineq_1}, \reff{ineq_2}, \reff{ineq_4} and \reff{ineq_5}, we get the expected result.
 \end{proof}
\subsection{Some results on the Gamma function}
We give here some results on the moments of  Gamma random variables.
\begin{lem}
  Let  $k,\ell,n\in(0,+\infty )$  and $\alpha,\beta,\gamma\in [0, +\infty )$  such  that
  $k+\ell+n+\alpha+\beta>\gamma$.                                    Let
  $\Gamma_{k},\Gamma_{\ell},\Gamma_{n}$   be  three   independent  Gamma
  random variables  with respective parameter $(k,1)$,  $(\ell,1)$ and $(n,1)$.
  Then we have:
 \[\E\left[\frac{\Gamma_{k}^{\alpha}\,\Gamma_{\ell}^{\beta}}{\left(\Gamma_{k}+\Gamma_{\ell}+\Gamma_{n}\right)^{\gamma}}\right]=
  \frac{\Gamma(k+\alpha)}{\Gamma(k)}\frac{\Gamma(\ell+\beta)}{\Gamma(\ell)}\frac{\Gamma(k+\ell+n+\alpha+\beta-\gamma)}{\Gamma(k+\ell+n+\alpha+\beta)}\cdot 
 \]
\end{lem}

\begin{proof}
Elementary computations give that for all  non negative function $f\in \mathcal{B}([0,1])$,
\[\E\left[\Gamma_{k}^{\alpha}f\left(\Gamma_{k}\right)\right]
=\E\left[\Gamma_{k}^{\alpha}\right]
\E\left[f\left(\Gamma_{k+\alpha}\right)\right]
=\frac{\Gamma(k+\alpha)}{\Gamma(k)} 
\E\left[f\left(\Gamma_{k+\alpha}\right)\right]. 
\]
We deduce that:
 \begin{align*}
  \E\left[\frac{\Gamma_{k}^{\alpha}\,\Gamma_{\ell}^{\beta}}
   {\left(\Gamma_{k}+\Gamma_{\ell}+\Gamma_{n}\right)^{\gamma}}\right]  
  &=\E\left[\E\left[\left.\frac{\Gamma_{k}^{\alpha}\,
    \Gamma_{\ell}^{\beta}}{\left(\Gamma_{k}+\Gamma_{\ell}+\Gamma_{n}\right)^{\gamma}}
    \right|  \Gamma_{\ell},\Gamma_{n}\right]\right]\\
  &=\E\left[\Gamma_{k}^{\alpha}\right]\E\left[\frac{\Gamma_{\ell}^{\beta}}
    {\left(\Gamma_{\ell}+\tilde\Gamma_{k+n+\alpha}\right)^{\gamma}}\right]\\ 
  &=\E\left[\Gamma_{k}^{\alpha}\right]\E\left[\Gamma_{\ell}^{\beta}\right]
    \E\left[\frac{1}{\Gamma_{k+\ell+n+\alpha+\beta}^{\gamma}}\right]\\ 
  &=\frac{\Gamma(k+\alpha)}{\Gamma(k)}\frac{\Gamma(\ell+\beta)}
    {\Gamma(\ell)}\frac{\Gamma(k+\ell+n+\alpha+\beta-\gamma)}{\Gamma(k+\ell+n+\alpha+\beta)}, 
 \end{align*}
 where  $\tilde\Gamma_{k+n+\alpha}$  is  a Gamma  random  variable  with
 parameter   $(k+n+\alpha,1)$   independent    of   $\Gamma_\ell$,   and
 $\Gamma_{k+\ell+n+\alpha+\beta}$  is  a   Gamma  random  variable  with
 parameter $(k+\ell+n+\alpha+\beta,1)$.
\end{proof}
We directly deduce the following result.
\begin{cor}\label{cor:moments_exp_gam}
Let $m\geq 2$. Let $(E_i, \, 1\le i \le m)$ be independent exponential 
random variables with parameter $1$ and
$S_{m}=\sum_{i=1}^{m}E_{i}$. Then for all $\alpha,\beta,\gamma\in [0,
+\infty )$ such that $m+\alpha+\beta>\gamma$, we have
\[
\E\left[\frac{E_{1}^{\alpha}E_{2}^{\beta}}{S_{m}^{\gamma}}\right]
=\Gamma(1+\alpha)\Gamma(1+\beta)\frac{\Gamma(m+\alpha+\beta-\gamma)}
{\Gamma(m+\alpha+\beta)}\cdot
\]
\end{cor}
\subsection{Elementary computations on the branch length of $\ct_{[n]}$}
We keep  notations from Section \ref{sec:premlim-lem}.   Recall that the
density of  $(h_{n,v}, v\in \tn)$  is, conditionally on $\tn$,  given by
\reff{eq:dens-h}. Recall $L_n=\sum_{v\in  \tn}h_{n,v}$ denotes the total
length of $  \ct_{[n]}$.  It is easy to
deduce that the density of $L_n$, conditionally on $\tn$, is given by:
\begin{equation}
   \label{eq:dens-Ln}
f_{L_n}(x)= 2 \frac{\alpha^{n+1}}{n!} x^{2n+1} \expp{-\alpha x^2}
\ind_{\{x>0\}}.  
\end{equation}
In particular, the random variable $L_n$ is independent of $\tn$. 
The first two moments of $L_n$ are given by 
\begin{align}
\label{eq:moments_Ln1}
\E[L_{n}]=\frac{1}{\sqrt{\alpha}}\frac{\Gamma(n+\frac{3}{2})}{\Gamma(n+1)}=
\frac{n+1}{\sqrt{\alpha}}\frac{\Gamma(n+\frac{3}{2})}{\Gamma(n+2)}
\quad
\text{and}
\quad
\E[L_{n}^{2}]=\frac{n+1}{\alpha}\cdot
\end{align}
According to \cite{g:seirg}, we have that 
$ (n+1)^{s-1} \leq {\Gamma(n+s)}/{\Gamma(n+1)}\leq n^{s-1}$ 
for  $n\in \N^*$ and $s\in [0,1]$. Hence, we obtain:
\begin{align}
\label{eq:moments_Ln2}
\frac{1}{\sqrt{\alpha}}\frac{n+1}{\sqrt{n+2}}
\leq
\E[L_{n}]
\leq
\frac{\sqrt{n+1}}{\sqrt{\alpha}}
\quad
\text{and}
\quad
\Var(L_{n})\leq\inv{{\alpha}}\cdot
\end{align}
Using that $L_n=\sum_{v\in   \tn}h_{n,v}$ and that,  conditionally on
$\tn$, the random variables  $(h_{n,v}, v\in \tn)$ are exchangeable, we
deduce that $\E[h_{n,\emptyset}]=\E[L_n]/(2n+1)$ and thus:
\begin{align}
\label{mh1}
\frac{1}{2\sqrt{\alpha (n+1)}}
\leq
\E[h_{n,\emptyset}]
=
\frac{1}{2\sqrt{\alpha}}\frac{\Gamma(n+\frac{1}{2})}{\Gamma(n+1)}
\leq
\frac{1}{2\sqrt{\alpha n}}\cdot
\end{align}

We finish by a  result on the covariance of the  branch lengths, used in
Lemma~\ref{lem:A2}.      We  define
$Y_{n,v}=\sqrt{n}  (\E[h_{n,v}]-h_{n,v})$ for $  v\in \tn$. Notice that
$(Y_{n,v}, v\in \tn)$ has   an
exchangeable distribution conditionally on $\tn$.
\begin{lem}\label{moment_hauteurs} Let $n\in \N^*$.   We have:
 \begin{equation}\label{mh4}
   \left|\E[Y_{n,\emptyset}Y_{n,1}]\right|\le \frac{1}{8\alpha n}  \quad\text{ and } \quad
   \E[Y_{n,\emptyset}^{2}]\le \frac{1}{2\alpha}\cdot
 \end{equation}
\end{lem}

\begin{proof}
 Using Lemma~\ref{lem:hh} and its notations, and
 Corollary~\ref{cor:moments_exp_gam} and \reff{eq:moments_Ln1}, we have,
 with $\bt\in \T$ full binary such that $|\bt|=2n+1$:  
 \begin{align*}
 \E[h_{n,\emptyset}h_{n,1}]=\E\left[L_n^{2}\right]\E\left[\frac{E_{\emptyset}E_{1}}{S_\bt^{2}}\right]&=\frac{n+1}{\alpha}\frac{1}{2(n+1)(2n+1)}=\frac{1}{2\alpha}\frac{1}{2n+1},  
 \end{align*}
and 
\begin{align*}
 \E[h_{n,\emptyset}^2]
=\
\E\left[L_n^{2}\right]\E\left[\frac{E_{\emptyset}^2}{S_\bt^{2}}\right]
=\frac{n+1}{\alpha}\frac{1}{(n+1)(2n+1)}=\frac{1}{\alpha}\frac{1}{2n+1}   \cdot
 \end{align*}
The lemma is then a consequence of these equalities and \reff{mh1}.
\end{proof}

\subsection{A deterministic representation formula}
\label{sec:proofZ}

\begin{lem}
   \label{lem:sigma-moment}
   Let $h\in \cc_+([0,1])$. We
   have that for all $a> 0$:
\begin{equation}
   \label{eq:I=J}
2\int_0^1 ds  \int_0^{h(s)} dr \, \sigma_{r,s}(h)^{a}=a(a+1) \int_{[0,1]^2}
 |s'-s|^{a-1}\,m_h(s,s')\, dsds' .
\end{equation}
\end{lem}

\begin{proof}
In this proof only, we shall write $m(s,t)$ and $\sigma_{r,s}$
respectively for $m_h(s,t)$ and $\sigma_{r,s}(h)$. Recall that
 $\sigma_{r,s}=\int_{0}^{1}{dt}\, \ind_{\{m(s,t)\ge r\}}$. We deduce
 that  $\int_{0}^{1}dt\, m(s,t)=\int_{0}^{h(s)}dr\, \sigma_{r,s}$ for
 every $s\in [0,1]$. Hence, the result is obvious for $a=1$. 

If $g\in \mathcal{B}([0,1])$ is a non negative function such that
$x^2g\in \cc^2([0,1])$ or if $g=x^{a-1}$ for $a>0$, we set: 
\[
I(g)= \int_{0}^{1}ds\int_{0}^{h(s)}dr\;\sigma_{r,s}\, g(\sigma_{r,s})
\quad  \text{and}\quad
 J(g)= \int_{0<s<t<1}\,ds\,dt \:  [x^{2}g]^{''}(t-s) \; m(s,t).
\]
We  then  have  to  prove  that  $J(g)=I(g)$  for  $g=x^{a-1}$  for  all
$a>0$. First of all, remark  that if $f,g\in \mathcal{B}([0,1])$ are non
negative functions such that $x^2f, x^2g\in \cc^2([0,1]$, we have:
\begin{equation}
\label{eq:cont_I}
\lvert I(g) -I(f)\rvert 
\le \int_{0}^{1}{ds}\int_{0}^{h(s)}{dr}\;\sigma_{r,s}
\left| g(\sigma_{s,r}) -f(\sigma_{s,r})\right| 
\le \lVert g-f \rVert_{\infty} \norm{h}_\infty 
\end{equation}
and
\begin{equation}
\label{eq:cont_J}
\lvert J(g) -J(f)\rvert 
 \le \lVert (x^{2}g)^{''} - (x^{2}f)^{''}\rVert_{\infty}
 \int_{0<s<t<1}\!\!\! ds\,dt\: m(s,t)
\leq \lVert (x^{2}g)^{''} - (x^{2}f)^{''}\rVert_{\infty}  \norm{h}_\infty .
\end{equation}
The proof of $J(g)=I(g)$ when $g=x^{a-1}$ is divided in 3 steps. First
of all, we prove the result when $a \in \N^*$, which
gives the equality when $g$ is polynomial. Then we get the case when
$g\in \cc^2([0,1])$ by Bernstein's approximation. This gives the case
$a \geq 3$. Finally, we give the result for $a \in
(0,3)\backslash \{1, 2\}$. 

\subsection*{1st step} Let $g=x^{a-1}$  with $a \in \N^*$. We have:
\begin{align*}
 I(x^{a-1})
&=\int_{0}^{1}{ds}\int_{0}^{h(s)}dr{\left(\int_{0}^{1}{dt}\, \ind_{\{m(s,t)\ge
  r\}}\right)^{a}} \\ 
&=\int_{[0,1]^{a+1}}ds_{1}\dots
  ds_{a+1}\left(\int_{0}^{h(s_1)}{dr}\, \ind_{\{m(s_{1},s_{2})\ge
  r\}}\dots \ind_{\{m(s_{1},s_{a+1})\ge r\}}\right) \\ 
&=\int_{[0,1]^{a+1}}ds_{1}\dots
  ds_{a+1}\left(\min(m(s_{1},s_{2}),\dots,m(s_{1},s_{a+1})\right)\\
&=\int_{[0,1]^{a+1}}ds_{1}\dots ds_{a+1}\left(m\left(\min_{1\le i
  \le a+1} s_i,\max_{1\le i\le a+1} s_i\right)\right).
\end{align*}
We have:
\begin{multline}
\int_{[0,1]^{a+1}}ds_{1}\dots ds_{a+1}\left(m\left(\min_{1\le i
  \le a+1} s_i,\max_{1\le i\le a+1} s_i\right)\right) \\ 
\begin{aligned}
   &=a(a+1)\int_{0<s<t<1} \,ds \,dt\int_{[s,t]^{a-1}}ds_{1}\dots
  ds_{a-1} \; m(s,t) \\ 
   \label{eq:=}
&=a(a+1)\int_{0<s<t<1} m(s,t)(t-s)^{a -1} \,ds\,dt .
\end{aligned}
\end{multline}
This gives $I(x^{a-1})=J(x^{a-1})$.

\subsection*{2nd  step}  Let  $g\in  \cc^2([0,1])$  be  a  non  negative
function. For $n\in  \N$, we define the  associated Bernstein polynomial
$B_n(g)$ by:
\[
B_n(g)(x)=\sum_{k=0}^n\binom{n}{k}\,g(k/n)\, x^k(1-x)^{n-k}, \quad x\in [0,1].
\]
It is well known (see for instance,
Theorem $6.3.2$ in  \cite{p:iap}) that for every $k\in\N$  and for every
$f\in\mathcal{C}^{k}([0,1])$,
$\lim_{n\rightarrow         \infty}\lVert         f^{(k)}-B_{n}^{(k)}(f)
\rVert_{\infty}=0$.
Using                                                               that
$\lVert  (x^{2}B_n(g))^{''}-(x^{2}g)^{''}  \rVert_{\infty}  \le  2\lVert
B_n(g)-g    \rVert_{\infty}+4\lVert    B_{n}^{'}(g)-g^{'}\rVert_{\infty}
+\lVert               B_n^{''}(g)-g^{''}               \rVert_{\infty}$,
we deduce from   \reff{eq:cont_I}  and   \reff{eq:cont_J} that
$J(g)=I(g)$. 

\subsection*{3rd  step} Let   $g=x^{a-1}$  with  $a \in (0, 3)\backslash\{1,2\}$. We
approximate $g$ by  functions in $\mathcal{C}^2([0,1])$. For $\delta\in
(0,1)$, we define:
\[
g_{\delta}(x)= 
\begin{cases} 
 P_{\delta}(x)   &\mbox{if }  0\le x \le \delta,\\
   g(x) &\mbox{if } \delta\le x\le 1 ,
\end{cases}
\]
where 
 $P_{\delta}$ is the polynomial with degree $2$ such that
 $P_{\delta}(\delta)=g(\delta)=\delta^{a-1}$,
 $P_{\delta}^{'}(\delta)=g^{'}(\delta)=(a-1)\delta^{a-2}$ and
 $P_{\delta}^{''}(\delta)=g^{''}(\delta)=(a-1)(a-2)\delta^{a-3}$.
     
We shall prove that $\lim_{\delta\rightarrow 0}I(g_{\delta})=I(g)$.
We have:
\[
g_{\delta}^{''}(x)= \begin{cases}
 g^{''}(\delta)   &\mbox{if }  0\le x \le \delta,\\
g^{''}(x) &\mbox{if } \delta\le x\le 1 .
                      \end{cases}
\]
\begin{itemize}
\item   Assume   $a\in   (0,1)$.  Let   $h=g_{\gamma}-g_{\delta}$   with
  $\delta,\gamma \in(0,1)$ such that $\delta < \gamma$. It is easy to check
  that $h''\leq 0$ on  $[0,1]$. Since $h^{'}(1)=h(1)=0$ by construction,
  we deduce  that $h\le 0$  on $[0,1]$.   Hence, when $\delta$  tends to
  $0$,  the sequence  $(g_{\delta}, 0<\delta<1)$  is non  decreasing and
  converges on $(0,1]$ towards $g$.  By monotone convergence theorem, we
  get $\lim_{\delta\rightarrow 0}I(g_{\delta})=I(g)$.
\item  Assume $a\in(1,3)$.  Notice  that  $(g_{\delta}, 0<\delta<1)$  is
  uniformly  bounded  by a  constant.  Hence,  by dominated  convergence
  theorem,                 we                obtain                 that
  $\lim_{\delta\rightarrow 0}I(g_{\delta})=I(g)$.
\end{itemize}
We now prove  that $\lim_{\delta\rightarrow
  0}J(g_{\delta})=J(g)$. Remark that if $x\in (\delta,1]$,
$(x^{2}g_{\delta}(x))^{''}=(x^{2}g(x))^{''}$, and that there exists a
constant $C(a)$, which does not depend on $\delta$, such that for
all $x\in (0,\delta]$, we have $|(x^{2}g_{\delta}(x))^{''}|\leq
C(a)\delta^{a-1}$. We get that:
 \begin{align*}
\left|J(g_\delta)-J(g)\right|
&=\left|\int_{0<s<t<1}m(s,t)\left((x^{2}g_{\delta}(x))^{''} 
- (x^{2}g(x))^{''}\right)_{x=t-s}dsdt\right|\\
&\le  \norm{h}_\infty 
\,\int_{0}^{\delta}\left(\big|(x^{2}g_{\delta}(x))^{''}-
  (x^{2}g(x))^{''}\big|\right)_{x=r}dr\\ 
&\le  \norm{h}_\infty \,\int_{0}^{\delta}\big(a(a+1)r^{a-1} +
  C(a)\delta^{a-1}\big)dr.
\end{align*}
We deduce that $\lim_{\delta\rightarrow 0}J(g_{\delta})=J(g)$. Thanks to
the 2nd step,  we have $J(g_{\delta})=I(g_{\delta})$ for all $\delta\in
(0,1)$. Letting $\delta$ goes down to 0, we deduce that $J(g)=I(g)$.
\end{proof}

\subsection{Proof of the first part of Lemma \ref{lem:ZH} (finiteness of
  $Z_\beta^H$ and \reff{eq:EZH})}
\label{sec:H}

We use the setting of \cite{dlg:rtlpsbp} on Lévy trees. 
Let $H$ be the height function of a stable Lévy tree with branching
mechanism $\psi(\lambda)=\kappa \lambda^\gamma$, with $\gamma\in (1,2]$
and $\kappa>0$. 

Let  $\N$  be the  excursion  measure  of  the  height process  and  set
$\sigma=\inf\{s>0, \,  H(s)=0\}$ for  the duration  of the  excursion so
that:  $\N[1  -   \expp{-\lambda  \sigma}]=\psi^{-1}(\lambda)$  for  all
$\lambda>0$.   Let  $\N^{(a)}[\bullet]=\N[\bullet  | \sigma=a]$  be  the
distribution of the  excursion of the height process  with duration $a$.
In particular,  we shall  prove the result  of Lemma  \ref{lem:ZH} under
$\N^{(1)}$.  We recall that:
\[
\N[\bullet]
=\int_0^\infty  \hat \pi(da)\,  \N^{(a)}[\bullet]
\quad \text{with}\quad
\hat \pi(da)=\inv{\gamma \kappa^{1/\gamma} \Gamma((\gamma-1)/\gamma)} \,
\frac{da}{a^{1+\inv{\gamma}}}\cdot
\]

In this proof only, we shall write $m$ for $m_H$ defined by
\reff{eq:def-m}. 
We extend the definitions  \reff{eq:def-srs} and \reff{eq:zb} as follows:
\[
\sigma_{r,s}=\int_0^\sigma dt\,  \ind_{\{\min(s,t)\geq  r\}}
\quad\text{and}\quad
Z_\beta^H=\int_{0}^{\sigma}ds\int_{0}^{H(s)}dr\;\sigma_{r,s}^{\beta-1}\quad
\text{for $\beta>0$.}
\] 

The integral in $ds/\sigma$ in  $Z_\beta^H$ corresponds to taking a leaf
at random  in the Lévy  tree. Using  Bismut's decomposition of  the Lévy
tree,  see   Theorem  4.5   in  \cite{dlg:pfalt}   or  Theorem   2.1  in
\cite{ad:farplt}, it is well known  that, since $\psi'(0)=0$, then under
$\N[\sigma \bullet]$, the height  $H(U)$, with $U$ uniformly distributed
over  $[0,   \sigma]$,  is   ``distributed''  as  $\ch$   with  Lebesgue
``distribution''  on  $(0,  +\infty  )$.  It  also  implies  that  under
$\N[\sigma         \bullet]$,         the        random         variable
$\left(H(U),   (\sigma_{H(U)-r,   U},   r\in   [0,   H(U)])\right)$   is
``distributed''  as  $\left(\ch,  (S_t, t\in  [0,  \ch])\right)$,  where
$S=(S_t, t\geq 0)$ is a subordinator,
with Laplace exponent  say $\phi$, independent of $\ch$. \\

We  prove   \reff{eq:EZH}  and  get   as  a  direct   consequence  using
monotonicity,   that   $\N^{(1)}$-a.s.,    for   all   $\beta>1/\gamma$,
$Z_\beta^H$ is finite. Using that:
\begin{equation}
   \label{eq:phi}
(\psi^{-1})'(\lambda)= \N\left[\sigma \expp{-\lambda
    \sigma}\right]
= \N\left[\sigma \expp{- \lambda \sigma_{0, U}}\right]
= \E\left[\expp{-\lambda S_\ch}\right]
= \int_0^\infty dt \, \E\left[\expp{-\lambda S_t}\right]=
\inv{\phi(\lambda)},
\end{equation}
we deduce that:
\begin{equation}
   \label{eq:phi=}
\phi(\lambda)=\inv{(\psi^{-1})'(\lambda)}
=\gamma \kappa^{1/\gamma} \lambda^{(\gamma-1)/\gamma}.
\end{equation}
Notice in particular that $S_t$ is distributed as $t^{\gamma/(\gamma-1)}
S_1$. We shall need later in the proof the following computation:
\begin{equation}
   \label{eq:S1}
 \E\left[S_1^{-(\gamma-1)/\gamma}\right]
=\inv{\Gamma\left(\frac{\gamma-1}{\gamma}\right)} \int_0^\infty
  dt \, t^{-1/\gamma} \E\left[\expp{-t S_1}\right] 
=\inv{\kappa^{1/\gamma}(\gamma-1)\Gamma
  \left(\frac{\gamma-1}{\gamma}\right)}\cdot  
\end{equation}

We set  $\Lambda (\lambda)
=\N\left[Z_\beta^H\expp{-\lambda \sigma}\right]$ for $\lambda>0$.
Using Bismut's decomposition again, we get:
\begin{align*}
\Lambda(\lambda)
= \N\left[\sigma\int_0^{H(U)} dr\,  \sigma_{H(U)-r, U}^{\beta-1} \expp{-\lambda
  \sigma_{0,U}}\right]
&=\E\left[\int_0^\ch dr\, S^{\beta-1}_r \expp{-\lambda S_\ch}
\right]\\
&=\E\left[\int_0^\infty dt \int_0^t dr\, S^{\beta-1}_r \expp{-\lambda S_t}
\right]\\
&=\E\left[\int_0^\infty dt \int_0^\infty  dr\, S^{\beta-1}_r
  \expp{-\lambda S_{t+r} }
\right].
\end{align*}
We have:
\begin{align*}
\Lambda(\lambda)
&=\E\left[\int_0^\infty dt \expp{-\lambda S_{t} }
\right]\E\left[\int_0^\infty  dr\, S^{\beta-1}_r
  \expp{-\lambda S_{r} }
\right]\\
&= \inv{\phi(\lambda)} \E\left[S_1^{\beta-1} \int_0^\infty dr\, 
  r^{(\beta-1)\gamma/(\gamma-1) } \expp{-\lambda r^{\gamma/(\gamma-1)}
    S_1} \right] \\
&=    \inv{\phi(\lambda)} \E\left[S_1^{-(\gamma-1)/\gamma} \right] 
\lambda^{-\beta+(1/\gamma)} \, \, \frac{\gamma-1}{\gamma} \int_0^\infty  du\, u^{\beta-1-(1/\gamma)}
  \expp{-u},
\end{align*}
where we used 
that $S$ has stationary independent increments for the first equality,
\reff{eq:phi} and that  $S_r$ is distributed as $r^{\gamma/(\gamma-1)}
S_1$ for the second, and the change of variable $u=\lambda S_1
r^{\gamma/(\gamma-1)}$ for the last.
Then use \reff{eq:S1} and \reff{eq:phi=} to deduce that:
\begin{equation}
   \label{eq:L=}
\Lambda(\lambda)=\frac{\Gamma\left(\beta-\inv{\gamma}\right)}{\gamma^2
  \kappa^{2/\gamma}  \Gamma\left(\frac{\gamma-1}{\gamma}\right)} \,
\lambda^{-1-\beta + \frac{2}{\gamma}}.
\end{equation}

On the other hand, we set
$G(a)=\N^{(a)}[Z_\beta^H]$ so that:
\[
\Lambda(\lambda)= \int_0^\infty \hat \pi(da)\, G(a)\expp{-\lambda a}.
\]
We deduce from the scaling property of the height function that,
under $\N^{(a)}$, the random variable $\Big((H(s), s\in [0,a]), \, (\sigma_{r,s}; r\in [0, H(s)], s\in
[0,a])\Big) $ is  distributed as the random variable 
$\Big((a^{(\gamma-1)/\gamma} H(s/a), \, s\in [0,a]), (a\sigma_{r,s/a}; r\in
[0, a^{(\gamma-1)/\gamma} H(s/a) ], s\in
[0,a])\Big) $  under $\N^{(1)}$. This implies that $Z^H_\beta$ is 
under $\N^{(a)}$ distributed as $a^{\beta+1-1/\gamma}Z^H_\beta$ 
under $\N^{(1)}$. This gives $G(a)=a^{\beta+1-1/\gamma} G(1)$. We deduce
that:
\[
\Lambda(\lambda)= G(1) \int_0^\infty \hat \pi(da)\, a^{\beta+1-\inv{\gamma}}
\expp{-\lambda a}= G(1) \frac{\Gamma\left(\beta+1- \frac{2}{\gamma}\right)}
{\gamma \kappa^{1/\gamma} \Gamma\left(\frac{\gamma-1}{\gamma}\right)}\,
\lambda ^{-\beta-1 + \frac{2}{\gamma}}.
\]
Then use \reff{eq:L=} to get that for all $\beta>0$:
\[
\N^{(1)}[Z^H_\beta]=G(1)=\inv{\gamma \kappa^{1/\gamma}} \,
\frac{\Gamma\left(\beta- 
    \inv{\gamma}\right)} {\Gamma\left(\beta+1-\frac{2}{\gamma}\right)}\cdot
\]
This gives \reff{eq:EZH} and that
$\N^{(1)}$-a.s., for all $\beta>1/\gamma$, $Z_\beta^H$ is finite.\\

We  prove now  that  $\N^{(1)}$-a.s., for  all $\beta\in  (0,1/\gamma]$,
$Z_\beta^H$ is infinite. Let  $\beta\in(0, 1/\gamma]$. Let
$U$ be uniform on $[0, \sigma]$  under $\N$. According to the first part
of  the   proof,  we  deduce   from  the  Bismut's   decomposition  that
$\int_{0}^{H(U)}dr\;\sigma_{r,U}^{\beta-1}$           is,          under
$\N[\sigma        \bullet        |H(U)=t]$,        distributed        as
$\int_0^{t} dr\, S_r^{\beta-1}$. Thanks to \cite{b:lp} see Theorem 11 in
chapter  III  and  since  $S$   is  a  stable  subordinator  with  index
$(\gamma-1)/\gamma$,              we              have              that
$\limsup_{r\rightarrow      0+}      S_r     /h(r)>0$      a.s.      for
$h(r)=r^{\gamma/(\gamma-1)}     \log(|\log(r)|)^{-1/(\gamma-1)}$.     As
$\beta\in           (0,           1/\gamma]$,          we           have
$\int_0  dr   \,  h(r)^{\beta-1}=+\infty  $.  This   implies  that  a.s.
$\int_0    dr   \,    S_r^{\beta-1}=+\infty   $.     We   deduce    that
$\N$-a.e.             $ds$-a.e.             on             $[0,\sigma]$,
$\int_{0}^{H(s)}dr\;\sigma_{r,s}^{\beta-1}=+\infty  $.  This gives  that
$\N$-a.e.  $Z_\beta^H=+\infty  $. Then  use the  scaling to  deduce that
$\N^{(1)}$-a.s. $Z_\beta^H=+\infty $.
%\end{proof}

\bibliographystyle{abbrv}
\bibliography{biblio_1}
 
\end{document}